\documentclass[11pt]{article}

\usepackage[margin=1in]{geometry}
\usepackage{amsmath,amssymb,amsthm,mathtools,bm}
\usepackage{booktabs,array,tabularx}
\usepackage{multirow,makecell,threeparttable}
\usepackage{placeins}
\usepackage{microtype}
\usepackage{enumitem}
\usepackage[table]{xcolor}
\definecolor{LightGray}{gray}{0.92}
\usepackage{aliascnt}
\usepackage{tikz}
\PassOptionsToPackage{pagebackref}{hyperref}

\AtBeginDocument{%
  \renewcommand*{\backref}[1]{}%
  \renewcommand*{\backrefalt}[4]{%
    \ifcase #1\relax
    \or
      \space\textit{(Cited on page~#2.)}%
    \else
      \space\textit{(Cited on pages~#2.)}%
    \fi
  }%
}

\newtheorem{theorem}{Theorem}[section]
\newtheorem{lemma}[theorem]{Lemma}

\newtheorem{corollary}[theorem]{Corollary}
\newtheorem{remark}[theorem]{Remark}
\theoremstyle{definition}
\newtheorem{assumption}[theorem]{Assumption}

\newcommand{\R}{\mathbb{R}}
\newcommand{\E}{\mathbb{E}}

\newcommand{\norm}[1]{\left\lVert #1\right\rVert}
\newcommand{\ip}[2]{\left\langle #1,#2\right\rangle}
\newcommand{\op}{\mathrm{op}}
\newcommand{\RR}{\textsc{RR}}
\newcommand{\wtO}{\widetilde{\mathcal O}}

\newcommand{\abs}[1]{\left| #1\right|}
\usepackage[hidelinks]{hyperref}
\usepackage[nameinlink,capitalize,noabbrev]{cleveref}
\hypersetup{
  pdftitle={Centered Permutation Prefixes for SGD with Random Reshuffling: Sharp Rates, Holder Geometry, and Composite Proximal Extensions},
  pdfauthor={Jiaxiang Li}
}

\title{Centered Permutation Prefixes for SGD with Random Reshuffling:\\
Sharp Rates, H\"older Geometry, and Composite Proximal Extensions}
\author{Jiaxiang Li\thanks{Grado Department of Industrial and Systems Engineering, Virginia Tech.  \texttt{jasonljx@vt.edu}}}
\date{September 2026}

\begin{document}
\maketitle

\begin{abstract}

We study stochastic gradient descent with random reshuffling for finite sums
\[
F(x)=\frac1n\sum_{i=1}^n f_i(x).
\]
For fresh reshuffling with a constant component stepsize, if each $f_i$ has an $L$-Lipschitz gradient and the average $F$ is $\mu$-strongly convex with a Lipschitz-continuous Hessian, we prove the last-epoch rate
\[
\mathbb E[F(y_K)-F(x_\star)]
=\widetilde O\!\left(T^{-2}+n^2T^{-3}\right),
\qquad T=nK,
\]
matching the known quadratic lower bound in its $(n,K)$-dependence.
The components may be nonconvex, and no componentwise Hessian continuity or separate bounded-iterate assumption is required.
More generally, a $\nu$-H\"older-continuous average Hessian adds only
$\widetilde O(n^{1+\nu}T^{-2-2\nu})$, so every $\nu\ge 1/2$ preserves the quadratic rate.
Under convex components, a decreasing-stepsize result removes the large-epoch requirement and recovers the same two-term scale once $nK$ exceeds the condition-number scale.

We also analyze epoch-wise ProxRR for $\mathcal P=F+\psi$.
Writing $x^\dagger$ for the composite minimizer and
$\beta_\star=\|\nabla F(x^\dagger)\|$, we prove
\[
\mathbb E\|y_K-x^\dagger\|^2
=\widetilde O\!\left(
\frac{\beta_\star^2}{K^2}
+T^{-2}+n^2T^{-3}
+n^{1+\nu}T^{-2-2\nu}
\right).
\]
For $\nu\ge 1/2$, we show that the $\beta_\star^2/K^2$ splitting term is unavoidable and obtain a matching lower bound up to logarithms in the stated constant-stepsize regime.
\end{abstract}

\section{Introduction}
\label{sec:introduction}

SGD with random reshuffling (\RR) processes every component of a finite sum once per epoch in a uniformly random order and is a common implementation in modern machine-learning training.
It is often faster in practice than with-replacement SGD, but its analysis is more delicate because the label used at an inner iteration is statistically dependent on the iterates generated by the preceding labels in the same epoch.

For smooth strongly convex finite sums, earlier work established \(1/K^2\)-type convergence; see \cite{gurbuzbalaban2021,haochen2019,nagaraj2019,mishchenko2020,ahn2020,nguyen2021}.
The sharp dependence on the number \(n\) of components depends on second-order regularity.
Safran and Shamir \cite{safran2020} proved that constant-stepsize \RR, evaluated at the last epoch iterate, can incur
\begin{equation}
\Omega\!\left(\frac{1}{(nK)^2}+\frac{1}{nK^3}\right)
\label{eq:intro-quadratic-lb}
\end{equation}
even on one-dimensional quadratic finite sums.
Rajput, Gupta, and Papailiopoulos \cite{rajput2020} obtained a matching upper bound when the average objective is quadratic.
For non-quadratic objectives under componentwise Lipschitz-Hessian assumptions, the previously available second-order upper bound was \(\wtO(T^{-2}+n^3T^{-3})=\wtO((nK)^{-2}+K^{-3})\) \cite{haochen2019}; Ahn, Yun, and Sra \cite{ahn2020} explicitly recorded the remaining factor-\(n\) gap.
In this work, we close this gap by centering the partial sums associated with the first labels of the sampled permutation before applying a Taylor expansion.

\paragraph{Contributions.}
Our main contributions are as follows.
\begin{enumerate}[leftmargin=2.1em,label=(\roman*)]
    \item \textbf{Sharp strongly convex rate under an average Lipschitz Hessian.}
    Under component smoothness, strong convexity of the average, and Lipschitz continuity of the average Hessian, we prove
    \[
    \E\|y_K-x_\star\|^2
    =\wtO\!\left(T^{-2}+n^2T^{-3}\right).
    \]
    The components may be nonconvex, their Hessians need not be Lipschitz, and bounded iterates or gradients are not assumed separately.
    To the best of our knowledge, this is the first upper bound under these assumptions that matches the quadratic lower bound in \((n,K)\).

    \item \textbf{A H\"older-Hessian threshold at \(\nu=\tfrac12\).}
    Under \(\|\nabla^2F(x)-\nabla^2F(y)\|_{\op} \le\rho_\nu\|x-y\|^\nu\), the additional global term is \(\wtO(n^{1+\nu}T^{-2-2\nu})\).
    For every \(\nu\ge\tfrac12\), this term is dominated by \(n^2T^{-3}\).
    \Cref{fig:holder-threshold} visualizes the resulting upper-bound threshold in the \(K\)-decay exponent.

    \item \textbf{No large-epoch requirement under convex components.}
    Under component convexity and the standard bounded-iterates condition, we combine a per-update estimate of Ahn, Yun, and Sra~\cite{ahn2020} with our centered-prefix one-epoch bound.
    For any fixed \(\zeta>4\), a decreasing stepsize gives, for every \(K\ge1\),
    \[
    \E\|y_K-x_\star\|^2
    =O\!\left(
      \left(\frac{\kappa}{\kappa+nK}\right)^{\!\zeta}
      +(\kappa+nK)^{-2}
      +n^2(\kappa+nK)^{-3}
      +n^{1+\nu}(\kappa+nK)^{-2-2\nu}
    \right),
    \]
    with fixed problem parameters and the initial distance suppressed.
    For \(\nu\ge\tfrac12\) and \(nK\gtrsim\kappa\), this becomes \(O(T^{-2}+n^2T^{-3})\), even when only a few epochs are run.

    \item \textbf{A sharp composite ProxRR theorem.}
    For \(\mathcal P=F+\psi\), where \(\psi\) is proper, closed, and convex, and one proximal map is evaluated after each reshuffled pass, we prove
    \[
    \E\|y_K-x^\dagger\|^2
    =\wtO\!\left(
    \frac{\beta_\star^2}{K^2}+T^{-2}+n^2T^{-3}
    +n^{1+\nu}T^{-2-2\nu}\right).
    \]
    A zero-variance construction proves that the \(\beta_\star^2/K^2\) splitting term is unavoidable.
    For \(\nu\ge\tfrac12\), a direct-sum lower bound matches all three leading terms up to logarithms once the Safran--Shamir lower-bound truncation is inactive.
\end{enumerate}

\paragraph{Comparison at a glance.}
Table~\ref{tab:rates} compares the standard with-replacement SGD benchmark with the random-reshuffling upper and lower bounds most directly related to our main theorem.
As in Ahn, Yun, and Sra~\cite{ahn2020}, the objective classes are ordered from more general to more restrictive.
The composite results use a different error measure and are summarized separately in Table~\ref{tab:prox-rates}.

\begin{table}[!t]
\caption{\small
Comparison of with-replacement SGD and random reshuffling for smooth strongly convex finite sums.
All rates are for expected objective suboptimality.
The with-replacement SGD row applies to all three nested objective classes.
We set \(T=nK\), \(\kappa=L/\mu\), and suppress logarithms, fixed problem parameters, and initial-distance factors.
Here \RR{}, (BI), (B2), and (LB) denote random reshuffling, bounded iterates, a bounded stochastic-gradient second moment, and a lower bound, respectively.
}
\centering
\begin{threeparttable}
\setlength{\tabcolsep}{2.4pt}
\renewcommand{\arraystretch}{1.15}
\scriptsize
\begin{tabularx}{\textwidth}{
|>{\raggedright\arraybackslash}p{0.14\textwidth}
|>{\raggedright\arraybackslash}p{0.145\textwidth}
|>{\raggedright\arraybackslash}p{0.225\textwidth}
|>{\centering\arraybackslash}X
|>{\centering\arraybackslash}p{0.175\textwidth}|}
\hline
Objective class & Components & Sampling and reference & Convergence rate & Scope / assumptions\\
\hline\hline

\multirow{3}{*}{\makecell[l]{\(F\) strongly\\convex}}
& \(f_i\) smooth
& With-replacement SGD: Rakhlin et al.~\cite{rakhlin2012}\tnote{$\clubsuit$}
& \(O(T^{-1})\)
& last iterate; \(\eta_t\asymp t^{-1}\); (B2)\\
& \(f_i\) smooth
& \RR{}: Ahn et al.~\cite{ahn2020}\tnote{$\dagger$}
& \(\widetilde O(nT^{-2})\)
& \(K\gtrsim\kappa\)\\
& \(f_i\) smooth convex
& \RR{}: Rajput et al.~\cite{rajput2020}
& \(\Omega(nT^{-2})\) (LB)
& constant stepsize\\
\hline\hline

\multirow{3}{*}{\makecell[l]{\(F\) strongly\\convex; average\\\(\nu\)-H\"older Hessian}}
& \(f_i\) smooth convex
& \RR{}: HaoChen and Sra~\cite{haochen2019}\tnote{$\ddagger$}
& \(\widetilde O(T^{-2}+n^3T^{-3})\)
& \(\nu=1\), \(K\gtrsim\kappa^2\), (BI)\\
& \(f_i\) smooth
& \cellcolor{LightGray}\textbf{\RR{}: ours} (Cor.~\ref{cor:holder-sc-rate})
& \cellcolor{LightGray}\(\widetilde O(T^{-2}+n^2T^{-3})\)
& \cellcolor{LightGray}\(\nu\ge\tfrac12\), \(K\gtrsim\kappa^2\)\\
& \(f_i\) smooth convex
& \cellcolor{LightGray}\textbf{\RR{}: ours} (Cor.~\ref{cor:all-k-holder-rate})
& \cellcolor{LightGray}\makecell{\(O((\kappa/(\kappa+nK))^\zeta\)\\
\({}+(\kappa+nK)^{-2}\)\\
\({}+n^2(\kappa+nK)^{-3})\)}
& \cellcolor{LightGray}\(\nu\ge\tfrac12\), \(K\ge1\), (BI), varying steps\\
\hline\hline

\multirow{3}{*}{\makecell[l]{\(F\) strongly\\convex quadratic}}
& \(f_i\) smooth convex
& \RR{}: Rajput et al.~\cite{rajput2020}\tnote{*}
& \(\widetilde O(T^{-2}+n^2T^{-3})\)
& \(K\gtrsim\kappa^2\), (BI)\\
& \(f_i\) smooth convex
& \RR{}: Ahn et al.~\cite{ahn2020}\tnote{*}
& \(O(T^{-2}+n^2T^{-3})\)
& \(K\ge1\), (BI), varying steps\\
& \(f_i\) smooth quadratic convex
& \RR{}: Safran and Shamir~\cite{safran2020}
& \(\Omega(T^{-2}+n^2T^{-3})\) (LB)
& constant stepsize\\
\hline
\end{tabularx}
\begin{tablenotes}[flushleft]
\scriptsize
\item[$\clubsuit$] Rakhlin, Shamir, and Sridharan prove the optimal generic \(O(T^{-1})\) last-iterate rate for smooth strongly convex stochastic optimization under (B2).
The same generic upper bound applies to the nested H\"older-Hessian and quadratic classes; regularity of the average alone does not force the sampling variance at the optimum to vanish.
\item[$\dagger$] Ahn et al. state the general result for the best epoch iterate.
Under their bounded-iterates condition, the same order holds for the last epoch iterate.
\item[$\ddagger$] HaoChen and Sra additionally assume that every component is convex and has a Lipschitz-continuous Hessian.
\item[*] The quadratic-average upper bounds of Rajput et al. and Ahn et al. do not require the individual components \(f_i\) to be quadratic.
\end{tablenotes}
\end{threeparttable}
\label{tab:rates}
\end{table}
\FloatBarrier

\section{Related work}
\label{sec:related}

\paragraph{With-replacement benchmark.}
With independent uniform sampling, standard SGD for smooth strongly convex stochastic optimization attains the optimal generic \(O(T^{-1})\) expected objective rate for its last iterate with a \(t^{-1}\)-scale stepsize, under a bounded second-moment condition on the stochastic gradients~\cite{rakhlin2012}.
This is the conventional with-replacement baseline in the random-reshuffling literature \cite{haochen2019,rajput2020,ahn2020}.
The same upper benchmark applies to the nested classes in which the average has a H\"older-continuous Hessian or is quadratic: regularity of the average alone does not force the component-gradient variance at the optimum to vanish.

\paragraph{Finite-epoch theory and lower bounds.}
G\"urb\"uzbalaban, Ozdaglar, and Parrilo~\cite{gurbuzbalaban2021} gave foundational asymptotic explanations for reshuffling.
HaoChen and Sra~\cite{haochen2019} obtained the first nonasymptotic acceleration result under second-order assumptions, with rate \(\wtO(T^{-2}+n^3T^{-3})\).
Nagaraj, Jain, and Netrapalli~\cite{nagaraj2019}, Mishchenko, Khaled, and Richt\'arik~\cite{mishchenko2020}, Ahn, Yun, and Sra~\cite{ahn2020}, and Nguyen et al.~\cite{nguyen2021} broadened the assumptions and sharpened the rates.
Ahn et al. also showed that component convexity permits varying stepsizes that remove the lower bound on the number of epochs; their sharp two-term all-epoch theorem assumes that the average objective is quadratic.
Our all-epoch theorem uses their per-update estimate during the initial, larger-stepsize phase and then uses our centered-prefix one-epoch bound once the epoch stepsize enters its stable range.
The algorithm follows one prescribed decreasing schedule throughout; only the inequality used in the proof changes.
This replaces quadraticity of the average by H\"older continuity of its Hessian.
Liu~\cite{liu2026} subsequently established finite-epoch dominance of random reshuffling over with-replacement SGD in smooth convex optimization under broad stepsizes; that comparison theorem does not state the strongly convex last-iterate two-term rate considered here.
Safran and Shamir~\cite{safran2020} established the quadratic lower bound \eqref{eq:intro-quadratic-lb}; Rajput et al.~\cite{rajput2020} matched it for a multidimensional quadratic average and proved the stronger \(\Omega(nT^{-2})\) obstruction for the merely smooth class.
Cha, Lee, and Yun~\cite{cha2023} further refined lower bounds in the condition number and heterogeneity parameters.

\paragraph{Composite and nonsmooth models.}
Mishchenko, Khaled, and Richt\'arik~\cite{mishchenko2022} introduced ProxRR and FedRR, applying one common proximal map after each epoch.
Liu and Zhou~\cite{liu2024} obtained raw last-iterate objective bounds for the same epoch-end proximal template under convex smooth components.
Our composite proof uses nonexpansiveness to compare the post-epoch iterate with the optimal proximal fixed point; the centered-prefix expansion is then applied only to the smooth pre-proximal error.
This yields sharper reshuffling terms under average H\"older curvature and permits nonconvex components.

Different algorithms are needed when the component losses themselves are nonsmooth or when a proximal step is applied after every component.
Qiu, Li, and Milzarek~\cite{qiu2025normprr} analyze a normal-map proximal reshuffling method for nonsmooth nonconvex finite sums.
Liu and Zhou~\cite{liu2025nonsmooth} obtain improved last-iterate guarantees for nonsmooth convex shuffling, and Josz, Lai, and Li~\cite{josz2024proxrr} study locally Lipschitz summands through trajectory tracking.
These results are complementary to the epoch-wise smooth-component theorem in \cref{sec:proximal}.

\paragraph{Distributed methods, variational inequalities, and designed orderings.}
Huang, Zhou, and Pu~\cite{huang2023distributed} analyze gradient-tracking and exact-diffusion reshuffling with explicit spectral-gap dependence.
Emmanouilidis, Vidal, and Loizou~\cite{emmanouilidis2024} study stochastic extragradient with reshuffling for variational inequalities, while Chae, Yun, and Kim~\cite{chae2024} show that anchoring can be necessary in convex-concave problems.
GraB~\cite{lu2022grab} actively balances prefixes, and block-reshuffling and reversal schemes~\cite{nguyen2026block} reduce prefix-variance or order-dependent terms.
The abstract estimate \eqref{eq:abstract-transfer} explains how such moment improvements affect curvature bias, but not network-disagreement or operator-specific terms.

\paragraph{Organization.}
\Cref{sec:main-results} states the algorithm and the strongly convex results in the merely smooth and Lipschitz-Hessian regimes.
\Cref{sec:common-estimates} develops the common prefix, path, and conditional-mean estimates.
\Cref{sec:extensions} gives the H\"older-Hessian and all-epoch extensions, and \cref{sec:proximal} gives the composite ProxRR upper and lower bounds.
\Cref{sec:discussion} records the precise scope and limitations of each result.
Proofs are collected in Appendices~\ref{app:baseline-proofs}, \ref{app:extension-proofs}, and~\ref{app:proximal-proofs}.

\section{Problem formulation and main results}
\label{sec:main-results}

\subsection{Algorithm and assumptions}

Consider
\begin{equation}
F(x)=\frac{1}{n}\sum_{i=1}^n f_i(x),
\qquad x\in\R^d,
\label{eq:finite-sum}
\end{equation}
where \(n\ge2\).
At epoch \(k\in\{0,\ldots,K-1\}\), Random Reshuffling (\RR{}) draws an independent uniform permutation \(\pi_k\) of \([n]:=\{1,\ldots,n\}\) and performs
\begin{equation}
x_{k,i}=x_{k,i-1}-\eta\nabla f_{\pi_k(i)}(x_{k,i-1}),
\qquad i=1,\ldots,n.
\label{eq:rr-update}
\end{equation}
Set
\begin{equation}
y_k:=x_{k,0},
\qquad
y_{k+1}:=x_{k,n},
\qquad
T:=nK,
\qquad
h:=n\eta.
\label{eq:epoch-notation}
\end{equation}
Thus \(y_K\) is the last epoch iterate and \(h\) is the effective epoch stepsize.

We impose the following assumptions for \Cref{sec:main-results}.

\begin{assumption}[Component smoothness]\label{ass:smooth}
Every \(f_i:\R^d\to\R\) is differentiable and has an \(L\)-Lipschitz gradient:
\[
\norm{\nabla f_i(x)-\nabla f_i(y)}\le L\norm{x-y}
\qquad\text{for all }x,y\in\R^d.
\]
The components need not be convex.
\end{assumption}

\begin{assumption}[Strong convexity of the average]\label{ass:sc}
The average \(F\) is \(\mu\)-strongly convex, with unique minimizer \(x_\star\).
Hence \(\nabla F(x_\star)=0\) and \(0<\mu\le L\).
\end{assumption}

The stronger regime additionally uses the following condition.

\begin{assumption}[Lipschitz Hessian of the average]\label{ass:hess}
The function \(F\) is twice differentiable and
\[
\norm{\nabla^2F(x)-\nabla^2F(y)}_{\op}
\le \rho\norm{x-y}
\qquad\text{for all }x,y\in\R^d.
\]
No corresponding condition is imposed on the individual component Hessians.
\end{assumption}

Define
\begin{equation}
G_\star:=\max_{i\in[n]}\norm{\nabla f_i(x_\star)},
\qquad
D:=\max\left\{\norm{y_0-x_\star},\frac{G_\star}{2L}\right\},
\qquad
G:=G_\star+2LD.
\label{eq:DG}
\end{equation}
When \(D=0\), every component gradient vanishes at \(y_0=x_\star\), and the algorithm remains at the optimum.
We therefore focus on \(D>0\).
Observe that
\begin{equation}
G_\star\le2LD,
\qquad
G\le4LD.
\label{eq:G-4LD}
\end{equation}

The common within-epoch estimates only require the dimensionless path-stability condition
\begin{equation}
hL=n\eta L\le\frac1{32}.
\label{eq:path-stability}
\end{equation}
For the strongly convex distance recursions, we use the stronger condition
\begin{equation}
h=n\eta\le\frac{\mu}{32L^2},
\label{eq:stepsize}
\end{equation}
which implies \eqref{eq:path-stability} because \(\mu\le L\).

\paragraph{What is a permutation prefix?}
Fix one epoch and assume that its starting point is \(x\).
Write
\[
g_i:=\nabla f_i(x),
\qquad
g:=\frac1n\sum_{i=1}^n g_i=\nabla F(x).
\]
For \(m\in\{0,\ldots,n\}\), the first \(m\) labels of the sampled permutation form the prefix evaluated at the epoch start
\[
P_m:=\sum_{j=1}^m g_{\pi(j)}.
\]
A direct path-wise estimate uses \(\|P_m\|\le m\max_i\|g_i\|\), so summing \(\|P_m\|^2\) over an epoch costs order \(n^3\).
This scaling is consistent with the \(n^3T^{-3}\) term in the earlier second-order analysis of HaoChen and Sra \cite{haochen2019}.
We instead decompose
\begin{equation}
P_m=mg+Z_m,
\qquad
Z_m:=\sum_{j=1}^m(g_{\pi(j)}-g),
\qquad
\E Z_m=0.
\label{eq:intro-center}
\end{equation}
The term \(mg\) is the sum that would be obtained by replacing every selected component gradient by the average gradient \(g\).
For the remaining zero-mean term, sampling without replacement gives the exact identity
\[
\sum_{m=0}^{n-1}\E\|Z_m\|^2=\frac{n(n+1)}6V,
\qquad
V:=\frac1n\sum_{i=1}^n\|g_i-g\|^2,
\]
which is order \(n^2\), not \(n^3\).

The calculation above evaluates every component gradient at the starting point \(x\), whereas the algorithm does not.
Two comparisons are therefore needed.
First, the component at position \(m+1\) is chosen from the labels not used in the first \(m\) positions.
Its label is consequently dependent on \(P_m\), and \(\nabla f_{\pi(m+1)}\) cannot be replaced by \(\nabla F\) without an error term.
Lemma \ref{lem:coupling} bounds that error by \(\eta L\sqrt{2V}\).
Second, the algorithm evaluates \(\nabla f_{\pi(i)}\) at the actual point \(x_{i-1}\).
For comparison, we introduce \(\widetilde x_{i-1}\), obtained from the same preceding labels but with every preceding gradient evaluated at \(x\).
Lemma \ref{lem:B} bounds the difference between evaluating the current gradient at \(x_{i-1}\) and at \(\widetilde x_{i-1}\).

The same calculation also explains the H\"older-Hessian extension.
At the deterministic point \(a_m=x-\eta mg\), the reference iterate is \(a_m-\eta Z_m\).
If the average Hessian is \(\nu\)-H\"older continuous, then
\begin{equation}
\left\|\E\nabla F(a_m-\eta Z_m)-\nabla F(a_m)\right\|
\le
\frac{\rho_\nu\eta^{1+\nu}}{1+\nu}
\,\E\|Z_m\|^{1+\nu}.
\label{eq:intro-holder-bias}
\end{equation}
The first-order Taylor term has zero expectation because \(\E Z_m=0\).
The strongly convex extension combines this bias estimate with squared-distance contraction, while the ProxRR extension first uses nonexpansiveness of the proximal map and then applies the same smooth one-epoch analysis.

\subsection{Main theorems}

\begin{theorem}[Lipschitz-Hessian finite-horizon bound]\label{thm:lh-main}
Suppose Assumptions~\ref{ass:smooth}, \ref{ass:sc}, and~\ref{ass:hess} hold and \eqref{eq:stepsize} is satisfied.
Then, for every \(K\ge1\),
\begin{align}
\E\norm{y_K-x_\star}^2
\le{}&
\exp\!\left(-\frac{\mu\eta T}{2}\right)D^2
+\frac{64L^2G^2}{\mu^2}\eta^2
+\frac{6L^2G^2}{\mu}n^2\eta^3
+\frac{8\rho^2G^4}{9\mu^2}n^2\eta^4.
\label{eq:lh-global}
\end{align}
\end{theorem}
\paragraph{Interpretation.}
The four terms have different origins.
The exponential term is the contracted initial error.
The \(\eta^2\) term comes from the fact that the component used at a position is not independent of the preceding indices.
The \(n^2\eta^3\) term is the accumulated effect of evaluating component gradients at the changing inner iterates.
The final term is the nonlinear Taylor remainder controlled by the Lipschitz Hessian of the average objective.
See \Cref{sec:common-estimates} and Appendix~\ref{app:baseline-proofs-Lip-Hessian} for the proof.


\begin{corollary}[Optimal \((n,K)\)-rate with a Lipschitz Hessian]\label{cor:lh-rate}
Let \(T=nK\ge3\), choose
\begin{equation}
\eta=\frac{4\log T}{\mu T},
\label{eq:horizon-eta}
\end{equation}
and suppose
\begin{equation}
K\ge128\frac{L^2}{\mu^2}\log T.
\label{eq:epoch-condition}
\end{equation}
Then
\begin{align}
\E\norm{y_K-x_\star}^2
\le{}&
\frac{D^2}{T^2}
+\frac{1024L^2G^2\log^2T}{\mu^4T^2}
+\frac{384L^2G^2n^2\log^3T}{\mu^4T^3}
+\frac{2048\rho^2G^4n^2\log^4T}{9\mu^6T^4}.
\label{eq:lh-cor-explicit}
\end{align}
Consequently,
\begin{equation}
\E\bigl[F(y_K)-F(x_\star)\bigr]
=
\wtO\!\left(\frac{1}{(nK)^2}+\frac{1}{nK^3}\right),
\label{eq:lh-rate}
\end{equation}
where the asymptotic notation suppresses fixed problem parameters and logarithms.
\end{corollary}

\begin{theorem}[Merely smooth finite-horizon bound]\label{thm:smooth-main}
Suppose only Assumptions~\ref{ass:smooth} and~\ref{ass:sc} hold and \eqref{eq:stepsize} is satisfied.
Then, for every \(K\ge1\),
\begin{align}
\E\norm{y_K-x_\star}^2
\le{}&
\exp\!\left(-\frac{\mu\eta T}{2}\right)D^2
+\frac{16L^2G^2}{\mu^2}n\eta^2
+\frac{4L^2G^2}{\mu}n^2\eta^3.
\label{eq:smooth-global}
\end{align}
\end{theorem}
\paragraph{Interpretation.}
Without continuity of the Hessian, the conditional mean of the epoch error is controlled only by its standard deviation.
This replaces the leading \(\eta^2\) contribution in \cref{thm:lh-main} by \(n\eta^2\), which is exactly the factor of \(n\) lost in the merely smooth class.
See \Cref{sec:common-estimates} and Appendix~\ref{app:baseline-proofs-smooth} for the proof sketch and the detailed proof.


\begin{corollary}[Optimal \((n,K)\)-rate in the merely smooth class]\label{cor:smooth-rate}
Under \eqref{eq:horizon-eta}--\eqref{eq:epoch-condition},
\begin{align}
\E\norm{y_K-x_\star}^2
\le{}&
\frac{D^2}{T^2}
+\frac{256L^2G^2n\log^2T}{\mu^4T^2}
+\frac{256L^2G^2n^2\log^3T}{\mu^4T^3}.
\label{eq:smooth-cor-explicit}
\end{align}
In particular,
\begin{equation}
\E\bigl[F(y_K)-F(x_\star)\bigr]
=
\wtO\!\left(\frac{n}{T^2}\right)
=
\wtO\!\left(\frac{1}{nK^2}\right).
\label{eq:smooth-rate}
\end{equation}
\end{corollary}

The main text next states the estimates that make the two rates different and explains how they are combined.
All algebraic proofs are deferred to Appendix~\ref{app:baseline-proofs}.
Whenever one epoch is analyzed, the expectation is conditional on its starting point and is taken only over the fresh permutation; the tower property is used when epochs are combined.

\section{Proof ingredients for the main results}
\label{sec:common-estimates}

Every upper bound in \Cref{sec:main-results} uses the same three ingredients.
First, a full-gradient step decreases the relevant error measure.
Second, sampling without replacement gives an exact second moment for the partial sums evaluated at the epoch-start point.
Third, the actual inner iterates are compared with the iterates obtained from those partial sums.
The merely smooth and Hessian-regular proofs differ only in how they bound the conditional mean of the resulting epoch error.

\subsection{Consequences of strong convexity and smoothness}

\begin{lemma}[Basic smooth strongly convex inequalities]
\label{lem:basic-sc-smooth}
Under Assumptions~\ref{ass:smooth}--\ref{ass:sc}, the average objective is \(L\)-smooth and, for every \(x\in\R^d\),
\begin{align}
\norm{\nabla F(x)}
&\le L\norm{x-x_\star},
\label{eq:g-L-e}\\
\ip{x-x_\star}{\nabla F(x)}
&\ge
\frac{\mu L}{\mu+L}\norm{x-x_\star}^2
+\frac{1}{\mu+L}\norm{\nabla F(x)}^2                         \notag\\
&\ge
\frac{\mu}{2}\norm{x-x_\star}^2
+\frac{1}{2L}\norm{\nabla F(x)}^2,
\label{eq:interpolation}\\
\frac{\mu}{2}\norm{x-x_\star}^2
&\le F(x)-F(x_\star)
\le\frac{L}{2}\norm{x-x_\star}^2.
\label{eq:function-distance}
\end{align}
\end{lemma}
These inequalities are the deterministic part of the analysis and are standard results in convex and smooth optimization.
They show, in particular, that a full-gradient step supplies a negative term proportional to \(\|\nabla F(x)\|^2\), which absorbs state-dependent pieces of the reshuffling error.
See Appendix~\ref{app:baseline-proofs-shared} for a concise proof.


\subsection{A deterministic invariant ball}

\begin{lemma}[Trajectory confinement]\label{lem:confinement}
Suppose an epoch starts at \(x\) with \(\norm{x-x_\star}\le D\), and suppose \eqref{eq:stepsize} holds.
For every permutation and every inner iterate of that epoch,
\begin{equation}
\norm{x_i-x_\star}\le2D,
\qquad
\norm{\nabla f_j(x_i)}\le G
\quad(i=0,\ldots,n,\ j\in[n]),
\label{eq:inner-bound}
\end{equation}
and the next epoch starts in the smaller ball:
\begin{equation}
\norm{x_n-x_\star}\le D.
\label{eq:epoch-bound}
\end{equation}
Consequently, these bounds hold deterministically for every epoch.
\end{lemma}
\paragraph{Interpretation.}
The theorem does not assume bounded iterates or bounded component gradients.
This lemma derives both bounds from the initial distance, component smoothness, and a sufficiently small epoch stepsize.
It also guarantees that every epoch starts in the same radius-\(D\) ball, so the constants used later are uniform across epochs.
See Appendix~\ref{app:baseline-proofs-shared} for the proof.


\subsection{Partial sums evaluated at the epoch start}

Fix one epoch and condition on its starting point \(x\).
Throughout this epoch, write
\begin{equation}
g_i:=\nabla f_i(x),
\qquad
g:=\frac1n\sum_{i=1}^n g_i=\nabla F(x),
\qquad
V:=\frac1n\sum_{i=1}^n\norm{g_i-g}^2.
\label{eq:fixed-epoch-general}
\end{equation}
In the strongly convex sections, we additionally write \(e:=x-x_\star\).
By Lemma \ref{lem:confinement},
\begin{equation}
\norm e\le D,
\qquad
\norm{g_i}\le G,
\qquad
V\le G^2.
\label{eq:fixed-bounds}
\end{equation}
For a uniform permutation \(\pi\), define only the centered prefix
\begin{equation}
Z_m:=\sum_{j=1}^m(g_{\pi(j)}-g),
\qquad m=0,\ldots,n.
\label{eq:centered-prefix}
\end{equation}
Thus the corresponding uncentered prefix is simply \(mg+Z_m\), and \(\E Z_m=0\).

\begin{lemma}[Finite-population prefix moments]\label{lem:prefix}
For every \(0\le m\le n\),
\begin{equation}
\E\norm{mg+Z_m}^2
=m^2\norm g^2+\frac{m(n-m)}{n-1}V.
\label{eq:prefix-exact}
\end{equation}
Consequently,
\begin{align}
\sum_{m=0}^{n-1}\E\norm{mg+Z_m}^2
&=\frac{n(n-1)(2n-1)}{6}\norm g^2
 +\frac{n(n+1)}{6}V\notag\\
&\le\frac13\bigl(n^3\norm g^2+n^2V\bigr),
\label{eq:prefix-sum}
\\
\sum_{m=0}^{n-1}\E\norm{Z_m}^2
&=\frac{n(n+1)}{6}V.
\label{eq:centered-prefix-sum}
\end{align}
\end{lemma}
\paragraph{Interpretation.}
The identity separates the squared size of the partial sum into a deterministic part, \(m^2\|g\|^2\), and a sampling-without-replacement variance term.
After summing over all positions, the centered part costs order \(n^2V\), rather than the order \(n^3\) obtained from a worst-case bound on the uncentered sum.
See Appendix~\ref{app:baseline-proofs-shared} for the proof.


\subsection{From epoch-start gradients to the actual inner iterates}

Let the displacement and cumulative within-epoch error be
\begin{equation}
d_m:=x_m-x,
\qquad
R_m:=\sum_{j=1}^m
\bigl[\nabla f_{\pi(j)}(x_{j-1})-g_{\pi(j)}\bigr],
\qquad R_0=0.
\label{eq:path-defs}
\end{equation}
The update gives the exact identity
\begin{equation}
d_m=-\eta\bigl(mg+Z_m+R_m\bigr),
\label{eq:path-exact}
\end{equation}
and each summand in \(R_m\) has norm at most \(L\norm{d_{j-1}}\).

\begin{lemma}[Squared displacement and epoch-error variance]\label{lem:path-energy}
Under \eqref{eq:path-stability},
\begin{align}
\sum_{m=0}^{n-1}\E\norm{d_m}^2
&\le\eta^2\bigl(n^3\norm g^2+n^2V\bigr),
\label{eq:path-energy-bound}\\
\E\norm{R_n}^2
&\le L^2\eta^2\bigl(n^4\norm g^2+n^3V\bigr).
\label{eq:R-second-bound}
\end{align}
\end{lemma}
\paragraph{Interpretation.}
The first inequality bounds the total squared displacement of the actual inner iterates from the epoch-start point.
The second bounds the accumulated error made by evaluating \(\nabla f_{\pi(j)}\) at \(x_{j-1}\) instead of at \(x\).
These estimates are valid even though the permutation sum and the actual inner iterates are dependent.
See Appendix~\ref{app:baseline-proofs-shared} for the proof.


For the rest of the proof, abbreviate the full epoch error by
\begin{equation}
R:=R_n.
\label{eq:R-def}
\end{equation}
The epoch endpoint is
\begin{equation}
x_n=x-hg-\eta R.
\label{eq:endpoint-R}
\end{equation}
The variance bound \eqref{eq:R-second-bound} is common to both regularity regimes.
The distinction is how sharply one can control the mean \(\E R\).

\subsection{The merely smooth regime}
\label{sec:smooth-regime}

This section states the one-epoch estimate behind \cref{thm:smooth-main} and Corollary \ref{cor:smooth-rate} and explains why the absence of Hessian regularity loses a factor of \(n\).

\begin{lemma}[A universal mean-error bound]\label{lem:smooth-mean}
Under \eqref{eq:stepsize}, we can bound the expectation of \eqref{eq:R-def} as
\begin{equation}
\norm{\E R}
\le L\eta\bigl(n^2\norm{g}+n^{3/2}\sqrt V\bigr).
\label{eq:smooth-mean}
\end{equation}
\end{lemma}
\paragraph{Interpretation.}
A general estimate without additional regularity follows from Jensen's inequality: the norm of the mean is bounded by its root mean square.
Applying Jensen's inequality to Lemma~\ref{lem:path-energy} and separating the two terms under the square root gives the displayed bound.

\begin{lemma}[One-epoch recursion without Hessian regularity]\label{lem:smooth-one-epoch}
Condition on an epoch start \(x\) with \(\norm{x-x_\star}\le D\).
Under \eqref{eq:stepsize},
\begin{align}
\E\bigl[\norm{x_n-x_\star}^2\mid x\bigr]
\le{}&
\left(1-\frac{\mu h}{2}\right)\norm{e}^2
+\frac{8L^2}{\mu}n^2\eta^3V
+2L^2n^3\eta^4V.
\label{eq:smooth-one-epoch}
\end{align}
\end{lemma}
\paragraph{Interpretation.}
One complete reshuffled pass contracts the current squared distance by a fixed fraction of \(\mu h\), up to two additive variance terms.
This is the precise one-epoch inequality needed to prove \cref{thm:smooth-main}.
See Appendix~\ref{app:baseline-proofs-smooth} for the proof.


The combination of Lemma \ref{lem:confinement} and \ref{lem:smooth-one-epoch} gives \cref{thm:smooth-main}; iterating its affine recursion and substituting the horizon-dependent stepsize gives Corollary \ref{cor:smooth-rate}.
The complete argument is in Appendix~\ref{app:baseline-proofs-smooth}.

\subsection{The Lipschitz-Hessian regime}
\label{sec:lh-regime}

We now add Assumption~\ref{ass:hess}.
The second-moment estimate \eqref{eq:R-second-bound} is unchanged.
The new assumption is used only to sharpen the conditional mean \(\E R\).
For the first \(m\) labels in the permutation, all component gradients are temporarily evaluated at the common epoch-start point \(x\).
Their sum is \(mg+Z_m\), where \(mg\) is deterministic after conditioning on \(x\) and \(\E[Z_m\mid x]=0\).
We Taylor-expand only with respect to \(Z_m\), so its linear contribution vanishes after conditional expectation.

\subsubsection{A reference iterate with all preceding gradients evaluated at the epoch start}

The actual iterate before position \(i\) is
\[
 x_{i-1}=x-\eta\sum_{j=1}^{i-1}
 \nabla f_{\pi(j)}(x_{j-1}).
\]
For comparison, define
\begin{equation}
\widetilde x_{i-1}
:=x-\eta\sum_{j=1}^{i-1}\nabla f_{\pi(j)}(x)
=x-\eta\bigl((i-1)g+Z_{i-1}\bigr).
\label{eq:reference-point}
\end{equation}
Thus \(\widetilde x_{i-1}\) uses the same component labels as the algorithm, but every preceding component gradient is evaluated at the single point \(x\), rather than at the changing iterates \(x_0,x_1,\ldots\).

Recall that
\[
R=\sum_{i=1}^n
\bigl[\nabla f_{\pi(i)}(x_{i-1})-\nabla f_{\pi(i)}(x)\bigr].
\]
Insert \(\nabla f_{\pi(i)}(\widetilde x_{i-1})\) into every summand and define
\begin{align}
R^{\mathrm{ref}}
&:=\sum_{i=1}^n
\bigl[\nabla f_{\pi(i)}(\widetilde x_{i-1})
      -\nabla f_{\pi(i)}(x)\bigr],
\label{eq:R-reference}\\
R^{\mathrm{eval}}
&:=\sum_{i=1}^n
\bigl[\nabla f_{\pi(i)}(x_{i-1})
      -\nabla f_{\pi(i)}(\widetilde x_{i-1})\bigr].
\label{eq:R-evaluation}
\end{align}
Then \(R=R^{\mathrm{ref}}+R^{\mathrm{eval}}\).
The first vector can be analyzed from the permutation sums evaluated at \(x\).
The second vector is only the difference caused by using \(x_{i-1}\) rather than \(\widetilde x_{i-1}\) as the evaluation point for the current component.

\begin{lemma}[Evaluation-point correction]\label{lem:B}
Under \eqref{eq:path-stability},
\begin{equation}
\E\norm{R^{\mathrm{eval}}}^2
\le\frac12L^4\eta^4
\bigl(n^6\norm g^2+n^5V\bigr).
\label{eq:B-bound}
\end{equation}
\end{lemma}
\paragraph{Interpretation.}
The reference iterate \(\widetilde x_{i-1}\) uses the same first \(i-1\) component labels as the algorithm, but evaluates every one of those component gradients at the common point \(x\).
This lemma bounds the error caused by replacing the actual evaluation point \(x_{i-1}\) by \(\widetilde x_{i-1}\).
Its order \(\eta^4\) is small enough not to determine the final rate.
See Appendix~\ref{app:baseline-proofs-Lip-Hessian} for the proof.


\subsubsection{Dependence between the component at position \texorpdfstring{$i$}{i} and the preceding indices}

At position \(i\), the component label \(\pi(i)\) is sampled from the labels not used in positions \(1,\ldots,i-1\).
Consequently it is dependent on the reference point \(\widetilde x_{i-1}\), which is built from those earlier labels.
The next lemma bounds the error caused by this dependence.

\begin{lemma}[Conditional-label coupling]\label{lem:coupling}
For every position \(i\),
\begin{equation}
\left\lVert
\E\nabla f_{\pi(i)}(\widetilde x_{i-1})
-
\E\nabla F(\widetilde x_{i-1})
\right\rVert
\le\eta L\sqrt{2V}.
\label{eq:coupling}
\end{equation}
\end{lemma}
\paragraph{Interpretation.}
At position \(i\), the label \(\pi(i)\) is selected from the components not used in the first \(i-1\) positions.
Therefore the current component and the reference point \(\widetilde x_{i-1}\) are dependent.
The lemma quantifies the error in replacing the expected current component gradient by the expected full gradient at that same random point.
See Appendix~\ref{app:baseline-proofs-Lip-Hessian} for the proof.


\subsubsection{Taylor expansion around a deterministic point}

For \(m=0,\ldots,n-1\), define the deterministic point
\begin{equation}
a_m:=x-\eta m g.
\label{eq:mean-prefix-point}
\end{equation}
The reference iterate satisfies \(\widetilde x_m=a_m-\eta Z_m\).
Thus the only random displacement from \(a_m\) is the zero-mean vector \(-\eta Z_m\).
The next lemma states the resulting bound on the conditional mean of \(R^{\mathrm{ref}}\).
Its proof, including the Taylor expansion, is in Appendix~\ref{app:baseline-proofs-Lip-Hessian}.

\begin{lemma}[Conditional mean of the reference error]\label{lem:mean-A}
Define
\begin{align}
r_{\mathrm{lab}}
&:=
\sum_{m=0}^{n-1}
\left[
\E\nabla f_{\pi(m+1)}(\widetilde x_m)
-
\E\nabla F(\widetilde x_m)
\right],\\
r_{\mathrm{curv}}
&:=
\sum_{m=0}^{n-1}
\left[
\E\nabla F(\widetilde x_m)
-
\nabla F(a_m)
\right], \\
b&:=\sum_{m=0}^{n-1}\bigl[\nabla F(a_m)-g\bigr].
\label{eq:M-def}
\end{align}

Then
\begin{equation}
\E R^{\mathrm{ref}}
=b+r_{\mathrm{lab}}+r_{\mathrm{curv}},
\label{eq:EA-decomp}
\end{equation}
with
\begin{align}
\norm b
&\le\eta L\frac{n(n-1)}2\norm g,
\label{eq:M-bound}\\
\norm{r_{\mathrm{lab}}}
&\le n\eta L\sqrt{2V},
\label{eq:delta-bound}\\
\norm{r_{\mathrm{curv}}}
&\le\frac{\rho\eta^2n(n+1)}{12}V
\le\frac{\rho\eta^2n^2}{6}V.
\label{eq:q-bound}
\end{align}
\end{lemma}
\paragraph{Interpretation.}
The conditional mean of \(R^{\mathrm{ref}}\) is written as three terms: \(b\) is the deterministic change of the full gradient along the points \(x-\eta m g\); \(r_{\mathrm{lab}}\) is the error from dependence of the current label on the preceding labels; and \(r_{\mathrm{curv}}\) is the nonlinear Taylor remainder generated by the zero-mean fluctuation \(Z_m\).
Again see Appendix~\ref{app:baseline-proofs-Lip-Hessian} for the complete proof.


The deterministic points \(a_m=x-\eta m g\) are handled only with gradient smoothness.
Hessian regularity is used after subtracting these points, on the zero-mean displacement \(-\eta Z_m\).
This order of operations prevents the deterministic vector \(mg\) from entering the nonlinear Taylor remainder.

\subsubsection{One-epoch contraction}

\begin{lemma}[One-epoch recursion with a Lipschitz Hessian]\label{lem:lh-one-epoch}
Condition on an arbitrary epoch start \(x\).
Under \eqref{eq:stepsize},
\begin{align}
\E\bigl[\norm{x_n-x_\star}^2\mid x\bigr]
\le{}&
\left(1-\frac{\mu h}{2}\right)\norm e^2
+32\frac{L^2V}{\mu}n\eta^3
+3L^2Vn^3\eta^4
+\frac{4\rho^2}{9\mu}V^2n^3\eta^5.
\label{eq:lh-one-epoch}
\end{align}
\end{lemma}
\paragraph{Interpretation.}
This is the sharpened one-epoch recursion.
Compared with Lemma \ref{lem:smooth-one-epoch}, the leading variance contribution is \(n\eta^3V\), rather than \(n^2\eta^3V\).
After summing over epochs, this is the factor-\(n\) improvement in the final theorem.
See Appendix~\ref{app:baseline-proofs-Lip-Hessian} for the proof.


Combining Lemma \ref{lem:confinement} and \ref{lem:lh-one-epoch} across epochs gives \cref{thm:lh-main} and Corollary \ref{cor:lh-rate}.
The complete proof, including the absorption of state-dependent terms, is in Appendix~\ref{app:baseline-proofs-Lip-Hessian}.

\section{Extensions: H\"older Hessians and all-epoch rates}
\label{sec:extensions}

This section changes the regularity or geometry assumptions while keeping the same estimates for the permutation partial sums and for the difference between \(x_{i-1}\) and \(\widetilde x_{i-1}\).
The bounds in Lemma \ref{lem:prefix}, \ref{lem:path-energy}, \ref{lem:B} and \ref{lem:coupling} use only component smoothness and \eqref{eq:path-stability}.
The new assumptions are used when those bounds are inserted into a one-epoch contraction inequality.

\subsection{H\"older continuity of the average Hessian}

\begin{assumption}[H\"older Hessian of the average]
\label{ass:holder-hess}
For some exponent \(\nu\in(0,1]\) and constant \(\rho_\nu\ge0\), the average objective is twice differentiable and
\begin{equation}
 \norm{\nabla^2F(u)-\nabla^2F(v)}_{\op}
 \le \rho_\nu\norm{u-v}^{\nu}
 \qquad\text{for all }u,v\in\R^d.
\label{eq:holder-hess}
\end{equation}
No H\"older condition is imposed on the component Hessians.
\end{assumption}

For later constants, set
\begin{equation}
        c_\nu:=\frac{2^{-(1+\nu)}}{1+\nu}.
\label{eq:c-nu}
\end{equation}

\begin{lemma}[H\"older bound for the centered gradient bias]
\label{lem:holder-curvature}
Suppose Assumption~\ref{ass:holder-hess} holds.
Conditional on an epoch start \(x\), let \(g=\nabla F(x)\), let \(V\) and \(Z_m\) be defined as in \eqref{eq:fixed-epoch-general} and \eqref{eq:centered-prefix}, and define \(a_m=x-\eta m g\).
Then
\begin{align}
&\left\lVert
\sum_{m=0}^{n-1}
\left(\E\nabla F(a_m-\eta Z_m)-\nabla F(a_m)\right)
\right\rVert                                                     \notag\\
&\hspace{3cm}\le
c_\nu\rho_\nu\eta^{1+\nu}
 n^{(3+\nu)/2}V^{(1+\nu)/2}.
\label{eq:holder-curvature-bias}
\end{align}
\end{lemma}
\paragraph{Interpretation.}
The lemma compares the full gradient at the deterministic point \(a_m\) with the expected full gradient at the random point \(a_m-\eta Z_m\).
H\"older continuity makes this difference proportional to \(\eta^{1+\nu}\E\|Z_m\|^{1+\nu}\).
The exact partial-sum identity in Lemma \ref{lem:prefix} then gives the displayed dependence on \(n\) and \(V\).
See Appendix~\ref{app:extension-proofs} for the proof.


\begin{remark}[The same calculation for another ordering]
\label{rem:abstract-transfer}
The Taylor calculation itself does not require a uniform permutation.
Suppose another ordering rule produces a random partial sum \(Q_m\) that can be written as \(Q_m=\bar q_m+Z_m\), where \(\bar q_m\) is deterministic after conditioning on the epoch start and \(\E Z_m=0\).
With \(a_m:=x-\eta\bar q_m\), expanding around \(a_m\) gives
\begin{equation}
 \left\lVert\sum_m
 \bigl(\E\nabla F(a_m-\eta Z_m)-\nabla F(a_m)\bigr)\right\rVert
 \le \frac{\rho_\nu\eta^{1+\nu}}{1+\nu}
       \sum_m\E\norm{Z_m}^{1+\nu}.
\label{eq:abstract-transfer}
\end{equation}
Thus smaller moments of \(Z_m\) reduce this curvature term.
This estimate is only one part of a convergence proof: one must also bound the dependence of the current component on the preceding labels and the difference between the actual iterate \(x_m\) and the corresponding reference iterate \(\widetilde x_m\).
\end{remark}

\subsection{Strong convexity with a H\"older Hessian}

The first extension weakens the Lipschitz-Hessian assumption in the original strongly convex theorem.

\begin{theorem}[Strongly convex H\"older-Hessian bound]
\label{thm:holder-sc}
Suppose Assumptions~\ref{ass:smooth}, \ref{ass:sc}, and~\ref{ass:holder-hess} hold and \eqref{eq:stepsize} is satisfied.
Then, for every \(K\ge1\),
\begin{align}
\E\norm{y_K-x_\star}^2
\le{}&
\exp\!\left(-\frac{\mu\eta T}{2}\right)D^2
+\frac{64L^2G^2}{\mu^2}\eta^2
+\frac{6L^2G^2}{\mu}n^2\eta^3                                  \notag\\
&+\frac{32c_\nu^2\rho_\nu^2G^{2+2\nu}}{\mu^2}
  n^{1+\nu}\eta^{2+2\nu}.
\label{eq:holder-sc-global}
\end{align}
\end{theorem}
\paragraph{Interpretation.}
The first three terms are the same as in the Lipschitz-Hessian theorem.
Only the curvature term changes, from order \(n^2\eta^4\) to \(n^{1+\nu}\eta^{2+2\nu}\) after summing over epochs.
See Appendix~\ref{app:extension-proofs} for the proof.


\begin{corollary}[The half-H\"older threshold]
\label{cor:holder-sc-rate}
Under the horizon choice \eqref{eq:horizon-eta} and epoch condition \eqref{eq:epoch-condition},
\begin{align}
\E\norm{y_K-x_\star}^2
\le{}&\frac{D^2}{T^2}
+\frac{1024L^2G^2\log^2T}{\mu^4T^2}
+\frac{384L^2G^2n^2\log^3T}{\mu^4T^3}                           \notag\\
&+\frac{2^{7+2\nu}}{(1+\nu)^2}
\frac{\rho_\nu^2G^{2+2\nu}n^{1+\nu}\log^{2+2\nu}T}
     {\mu^{4+2\nu}T^{2+2\nu}}.
\label{eq:holder-sc-explicit}
\end{align}
In particular,
\begin{equation}
\E[F(y_K)-F(x_\star)]
=\wtO\!\left(
T^{-2}+n^2T^{-3}+n^{1+\nu}T^{-2-2\nu}
\right).
\label{eq:holder-sc-rate}
\end{equation}
If \(\nu\ge\tfrac12\), the final term is dominated by \(n^2T^{-3}\), and hence
\begin{equation}
\E[F(y_K)-F(x_\star)]
=\wtO\!\left(T^{-2}+n^2T^{-3}\right).
\label{eq:holder-half-rate}
\end{equation}
This \((n,K)\)-dependence is minimax-sharp up to logarithms in the constant-stepsize, last-iterate model because the class contains all quadratic instances.
\end{corollary}
\begin{figure}[t]
\centering
\begin{tikzpicture}[x=7.2cm,y=1.35cm]
  \draw[->] (0,1.8) -- (1.08,1.8) node[right] {$\nu$};
  \draw[->] (0,1.8) -- (0,4.25) node[above] {\(K\)-decay exponent};
  \draw (0,2) -- (1,4) node[above left] {$2+2\nu$};
  \draw[dashed] (0,3) -- (1.02,3) node[right] {$3$};
  \draw[dashed] (0.5,1.8) -- (0.5,3);
  \foreach \x/\lab in {0/0,0.5/{1/2},1/1}
    \draw (\x,1.76) -- (\x,1.84) node[below=3pt] {$\lab$};
  \foreach \y in {2,3,4}
    \draw (-0.012,\y) -- (0.012,\y) node[left=3pt] {$\y$};
  \node[align=center,anchor=west] at (0.54,2.48)
    {curvature term decays at least\\as fast as $K^{-3}$};
\end{tikzpicture}
\caption{
Ignoring logarithms and fixed parameters, the H\"older-curvature term is \(n^{-1-\nu}K^{-2-2\nu}\).
Its \(K\)-decay exponent reaches the exponent \(3\) of the quadratic-instance term \(1/(nK^3)\) at \(\nu=1/2\); for larger \(\nu\), it is higher order than that term.
}
\label{fig:holder-threshold}
\end{figure}
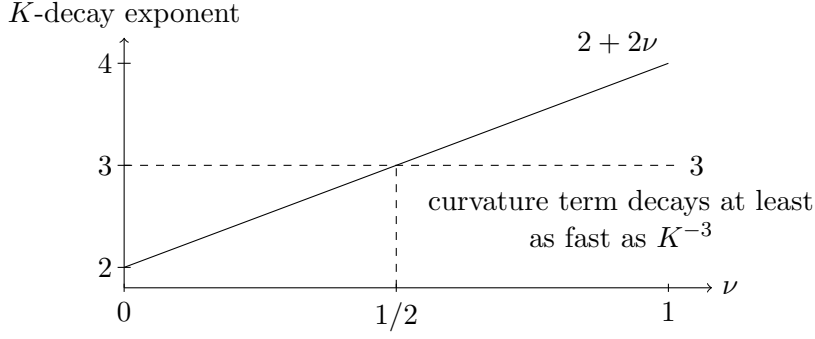

\subsection{Removing the large-epoch requirement under convex components}
\label{sec:all-k-holder}

The fixed-stepsize corollaries choose \(\eta=4\log(T)/(\mu T)\).
Their one-epoch analysis also requires \(n\eta\le\mu/(32L^2)\).
Substituting the horizon-tuned stepsize into this stability condition gives
\[
 K\ge128\kappa^2\log T.
\]
Thus, the large-epoch condition comes from using the same horizon-tuned stepsize from the first update onward.
When the components are convex, a per-update estimate remains valid for larger early steps.
We use a decreasing stepsize schedule so that this estimate controls the initial phase and our sharper centered-prefix estimate controls the later phase.

\begin{assumption}[Convex components and bounded iterates]
\label{ass:all-k-bounded}
Each \(f_i\) is convex.
For the varying-stepsize schedule below, assume that there is a compact set \(\mathcal X\subset\R^d\) containing \(x_\star\) and every iterate \(x_{k,i}\), uniformly over all epochs, inner positions, and realizations of the sampled permutations.
Since every \(\nabla f_i\) is continuous, the constant
\begin{equation}
 \bar G:=\max_{j\in[n]}\sup_{z\in\mathcal X}\|\nabla f_j(z)\|<\infty
\label{eq:all-k-Gbar}
\end{equation}
is well defined.
\end{assumption}

The following previously known estimate is used only during the initial phase.
\begin{lemma}[Per-update estimate of Ahn, Yun, and Sra]
\label{lem:ahn-per-iteration}
Under Assumptions~\ref{ass:smooth}, \ref{ass:sc}, and \ref{ass:all-k-bounded}, let \(x\) and \(x^+\) be two consecutive iterates of \RR{} and let the current stepsize be \(0<\gamma\le2/L\).
Then
\begin{equation}
 \E\|x^+-x_\star\|^2
 \le\left(1-\frac{\mu\gamma}{2}\right)\E\|x-x_\star\|^2
    +3\gamma^2\bar G^2+4\gamma^3\kappa L\bar G^2.
\label{eq:ahn-per-iteration}
\end{equation}
\end{lemma}
This is Proposition~D.1 in the supplementary material of Ahn, Yun, and Sra~\cite{ahn2020}.
Their Appendix~D explains why the standard bounded-iterates condition and component smoothness provide the finite constant \(\bar G\).

We now define the schedule.
Let \(\kappa=L/\mu\), fix \(\zeta>4\), and set
\begin{equation}
 q_0:=\zeta\kappa,
 \qquad q_k:=q_0+nk\quad(k\ge0).
\label{eq:all-k-q}
\end{equation}
The quantity \(q_k\) is a condition-number offset plus the number of component updates completed by the start of epoch \(k\).
During the first epoch, use
\begin{equation}
 \eta_{0,i}:=\frac{2\zeta}{\mu(q_0+i)},
 \qquad i=1,\ldots,n.
\label{eq:all-k-first-eta}
\end{equation}
For every later epoch \(k\ge1\), use one constant stepsize throughout that epoch:
\begin{equation}
 \eta_{k,i}:=\eta_k:=\frac{2\zeta}{\mu q_{k+1}},
 \qquad i=1,\ldots,n.
\label{eq:all-k-later-eta}
\end{equation}
The first-epoch steps decrease after every update.
Thereafter, the stepsize is constant within each epoch---as required by the centered-prefix analysis---and decreases from one epoch to the next.
For \(k\ge1\), we freeze the update-level \(1/t\)-type schedule at the end-of-epoch denominator \(q_{k+1}\).
Consequently, all \(n\) updates in epoch \(k\) use the same stepsize, as required by the centered-prefix one-epoch analysis.

There are two proof regimes.
If \(n\eta_k>\mu/(32L^2)\), the sharp centered-prefix recursion is not yet applicable, so we apply Lemma~\ref{lem:ahn-per-iteration} to the \(n\) updates separately.
Once \(n\eta_k\le\mu/(32L^2)\), we apply the H\"older one-epoch recursion derived in Appendix~\ref{app:extension-proofs}, equation~\eqref{eq:holder-sc-one-epoch}.
Both regimes lead to the same polynomial recursion after substituting \(\eta_k\asymp q_{k+1}^{-1}\).
A two-scale Chung-type lemma then sums that recursion over all epochs.

\begin{theorem}[All-epoch H\"older-Hessian bound]
\label{thm:all-k-holder}
Suppose Assumptions~\ref{ass:smooth}, \ref{ass:sc}, \ref{ass:holder-hess}, and~\ref{ass:all-k-bounded} hold.
Run \textnormal{\RR{}} with \eqref{eq:all-k-first-eta}--\eqref{eq:all-k-later-eta}.
Then there are constants \(C_0,C_1,C_2,C_\nu>0\), independent of \(n\) and \(K\), such that for every \(K\ge1\),
\begin{align}
 \E\|y_K-x_\star\|^2
 \le{}& C_0\left(\frac{q_0}{q_K}\right)^{\!\zeta}
          \|y_0-x_\star\|^2
       +\frac{C_1}{q_K^2}
       +\frac{C_2n^2}{q_K^3}
       +\frac{C_\nu n^{1+\nu}}{q_K^{2+2\nu}}.
\label{eq:all-k-holder-bound}
\end{align}
The constants depend only on \(\zeta,L,\mu,\rho_\nu,\nu\), and \(\bar G\).
\end{theorem}

\paragraph{How to read the bound.}
The theorem is valid for every \(K\ge1\); there is no condition of the form \(K\gtrsim\kappa^2\log T\).
Since \(q_K=\zeta\kappa+nK\), the denominators behave like powers of the total number of updates once \(nK\ge\zeta\kappa\).
For \(\nu\ge1/2\), the H\"older term is then dominated by \(n^2q_K^{-3}\), and the initialization term is of order \(q_K^{-2}\) after fixed problem parameters are suppressed.
This gives the following corollary.

\begin{corollary}[All-epoch bound and recovery of the two-term scale]
\label{cor:all-k-holder-rate}
Under the assumptions of Theorem~\ref{thm:all-k-holder}, if \(\nu\ge\tfrac12\), then
\begin{equation}
 \E\|y_K-x_\star\|^2
 =O\!\left(
   \left(\frac{q_0}{q_K}\right)^{\!\zeta}\|y_0-x_\star\|^2
   +q_K^{-2}+n^2q_K^{-3}
 \right)
 \qquad(K\ge1).
\label{eq:all-k-holder-sharp}
\end{equation}
If in addition \(T=nK\ge q_0\), then
\begin{equation}
 \E[F(y_K)-F(x_\star)]
 =O\!\left(T^{-2}+n^2T^{-3}\right).
\label{eq:all-k-holder-T-rate}
\end{equation}
Thus, when \(n\gtrsim\kappa\), the two-term scale already applies after a constant number of epochs.
The available matching lower bounds assume a constant stepsize, so \eqref{eq:all-k-holder-T-rate} does not establish minimax optimality over all varying schedules.
\end{corollary}

The complete proof is in Appendix~\ref{app:all-k-holder-proof}.

\section{Composite objectives and epoch-wise proximal reshuffling}
\label{sec:proximal}

In this section, we consider
\begin{equation}
    \min_{x\in\R^d}\;\mathcal P(x):=F(x)+\psi(x),
    \qquad
    F(x)=\frac1n\sum_{i=1}^n f_i(x),
\label{eq:composite-problem}
\end{equation}
where \(\psi:\R^d\to\R\cup\{+\infty\}\) is proper, closed, and convex.
We use \(\psi\), rather than \(h\), for the regularizer because \(h=n\eta\) already denotes the epoch stepsize.
The algorithm is the ``one-prox-per-epoch'' ProxRR method of Mishchenko, Khaled, and Richt\'arik \cite{mishchenko2022}:
\begin{align}
 x_{k,0}&=y_k,\notag\\
 x_{k,i}&=x_{k,i-1}-\eta\nabla f_{\pi_k(i)}(x_{k,i-1}),
             \qquad i=1,\ldots,n,\label{eq:proxrr-inner}\\
 y_{k+1}&=\operatorname{prox}_{h\psi}(x_{k,n}),
             \qquad h=n\eta,
\label{eq:proxrr-prox}
\end{align}
where
\[
\operatorname{prox}_{\tau\psi}(z)
:=\operatorname*{argmin}_{u}
\left\{\tau\psi(u)+\frac12\norm{u-z}^2\right\}.
\]
We use the standard nonexpansiveness property
\begin{equation}
\norm{\operatorname{prox}_{\tau\psi}(u)
      -\operatorname{prox}_{\tau\psi}(v)}
\le\norm{u-v}.
\label{eq:prox-nonexpansive}
\end{equation}

Throughout this section, each \(f_i\) is \(L\)-smooth but need not be convex, \(F\) is \(\mu\)-strongly convex and satisfies Assumption~\ref{ass:holder-hess}, and \(\psi\) is proper, closed, and convex.
The composite objective is therefore \(\mu\)-strongly convex and has a unique minimizer \(x^\dagger\).
Set
\begin{equation}
 g^\dagger:=\nabla F(x^\dagger),
 \qquad
 \beta_\star:=\norm{g^\dagger}.
\label{eq:composite-optimal-gradient}
\end{equation}
Composite optimality is equivalent to the proximal fixed-point identity
\begin{equation}
 -g^\dagger\in\partial\psi(x^\dagger),
 \qquad
 x^\dagger=\operatorname{prox}_{h\psi}(x^\dagger-hg^\dagger).
\label{eq:composite-fixed-point}
\end{equation}
Define
\begin{equation}
\begin{aligned}
 G_{\star,\mathrm c}
 &:=\max_{i\in[n]}\norm{\nabla f_i(x^\dagger)},\\
 D_{\mathrm c}
 &:=\max\left\{\norm{y_0-x^\dagger},
                  \frac{G_{\star,\mathrm c}}{2L}\right\},\\
 G_{\mathrm c}
 &:=G_{\star,\mathrm c}+2LD_{\mathrm c}.
\end{aligned}
\label{eq:composite-DG}
\end{equation}
We impose the stepsize condition
\begin{equation}
 h=n\eta\le\frac{\mu}{64L^2}.
\label{eq:prox-stepsize}
\end{equation}

The existing ProxRR analysis \cite{mishchenko2022} controls a shuffling radius that mixes centered gradient variance with \(n\norm{\nabla F(x^\dagger)}^2\).
Under convex components, the raw last-iterate objective result of Liu and Zhou \cite{liu2024} has, after suppressing fixed parameters and logarithms, the form
\begin{equation}
 \wtO\!\left(\frac{V_\star}{nK^2}
              +\frac{\beta_\star^2}{K^2}\right),
 \qquad
 V_\star:=\frac1n\sum_{i=1}^n
 \norm{\nabla f_i(x^\dagger)-g^\dagger}^2.
\label{eq:previous-prox-rate}
\end{equation}
Our result separates these two mechanisms: average curvature regularity improves the reshuffling contribution, but not the epoch-end splitting term \(\beta_\star^2/K^2\).

\begin{table}[!t]
\caption{\small
A summary of existing and new rates for epoch-wise ProxRR.
The metric is expected squared distance to the composite minimizer, except for the row marked $\ddagger$.
We set $T=nK$ and suppress logarithms and fixed problem parameters.
}
\centering
\begin{threeparttable}
\setlength{\tabcolsep}{3.8pt}
\renewcommand{\arraystretch}{1.14}
\footnotesize
\begin{tabularx}{\textwidth}{
|>{\raggedright\arraybackslash}p{0.25\textwidth}
|>{\raggedright\arraybackslash}p{0.22\textwidth}
|>{\centering\arraybackslash}X
|>{\centering\arraybackslash}p{0.19\textwidth}|}
\hline
Settings & References & Convergence rates & Assumptions\\
\hline\hline
$f_i$ smooth; $\psi$ convex
& Mishchenko et al.~\cite{mishchenko2022}
& $\widetilde O\!\left(\frac{V_\star}{nK^2}+\frac{\beta_\star^2}{K^2}\right)$
& each $f_i$ strongly convex\tnote{$\S$}\\
\hline
$f_i$ smooth convex; $\psi$ convex
& Liu and Zhou~\cite{liu2024}\tnote{$\ddagger$}
& $\widetilde O\!\left(\frac{V_\star}{nK^2}+\frac{\beta_\star^2}{K^2}\right)$
& raw last-iterate objective\\
\hline
\makecell[l]{$f_i$ smooth; average $\nu$-H\"older\\Hessian; $\psi$ convex}
& \cellcolor{LightGray}\textbf{Ours} (Cor.~\ref{cor:prox-rate})
& \cellcolor{LightGray}$\widetilde O\!\left(\frac{\beta_\star^2}{K^2}+T^{-2}+n^2T^{-3}\right)$
& \cellcolor{LightGray}$\nu\ge\tfrac12$, $K\gtrsim\kappa^2$\\
\hline
quadratic smooth part; $\ell_1$ regularizer
& \cellcolor{LightGray}\textbf{Ours} (Cor.~\ref{cor:prox-combined-lower})
& \cellcolor{LightGray}$\Omega\!\left(\frac{\beta_\star^2}{K^2}+T^{-2}+n^2T^{-3}\right)$ (LB)
& \cellcolor{LightGray}\makecell{constant stepsize;\\condition~\eqref{eq:prox-large-horizon-regime}}\\
\hline
\end{tabularx}
\begin{tablenotes}[flushleft]
\footnotesize
\item[$\S$] Mishchenko et al. also give a parallel result when the regularizer is strongly convex.
\item[$\ddagger$] Liu and Zhou measure the raw composite objective gap; the other rows use squared distance.
\item[] Here \(V_\star:=n^{-1}\sum_i\|\nabla f_i(x^\dagger)-\nabla F(x^\dagger)\|^2\) and \(\beta_\star:=\|\nabla F(x^\dagger)\|\).
\end{tablenotes}
\end{threeparttable}
\label{tab:prox-rates}
\end{table}
\FloatBarrier

\subsection{A proximal fixed-point recursion}

We use the standard pairwise interpolation inequality, which holds for arbitrary \(u,v\) and therefore does not require \(x^\dagger\) to minimize \(F\).

The purpose of this subsection is to compare one ProxRR epoch with the optimal proximal fixed point.
The smooth inner pass is analyzed before the proximal map is applied; nonexpansiveness then transfers that estimate to the post-proximal iterate.

\begin{lemma}[Pairwise interpolation]
\label{lem:pairwise-interpolation}
For every \(u,v\in\R^d\),
\begin{align}
 \ip{u-v}{\nabla F(u)-\nabla F(v)}
 &\ge
 \frac{\mu L}{\mu+L}\norm{u-v}^2
 +\frac{1}{\mu+L}\norm{\nabla F(u)-\nabla F(v)}^2              \notag\\
 &\ge
 \frac\mu2\norm{u-v}^2
 +\frac1{2L}\norm{\nabla F(u)-\nabla F(v)}^2.
\label{eq:pairwise-interpolation}
\end{align}
\end{lemma}
\paragraph{Interpretation.}
This is a standard result that follows from the same shifted-function co-coercivity argument as Lemma \ref{lem:basic-sc-smooth}; the proof is by applying the co-coercivity argument from Lemma \ref{lem:basic-sc-smooth} to the pair \((u,v)\), rather than to \((x,x_\star)\).

\begin{lemma}[Composite trajectory confinement]
\label{lem:prox-confinement}
If an epoch starts at \(x\) with \(\norm{x-x^\dagger}\le D_{\mathrm c}\), then every inner iterate satisfies
\begin{equation}
 \norm{x_i-x^\dagger}\le2D_{\mathrm c},
 \qquad
 \norm{\nabla f_j(x_i)}\le G_{\mathrm c},
\label{eq:prox-inner-bound}
\end{equation}
and the post-proximal output \(x^+:=\operatorname{prox}_{h\psi}(x_n)\) satisfies \(\norm{x^+-x^\dagger}\le D_{\mathrm c}\).
Thus these bounds hold in every epoch.
\end{lemma}
\paragraph{Interpretation.}
The smooth inner pass and the post-epoch proximal step remain in a deterministic bounded region.
Thus the composite theorem, like the smooth theorem, does not assume bounded iterates or gradients separately.
See Appendix~\ref{app:proximal-proofs} for the proof.


Fix one epoch and condition on its starting point \(x\).
We use
\begin{equation}
 e:=x-x^\dagger,
 \qquad
 \bar g:=\nabla F(x)-\nabla F(x^\dagger),
 \qquad
 g:=\nabla F(x),
\label{eq:prox-error-defs}
\end{equation}
and
\begin{equation}
 R:=\sum_{i=1}^n
 \bigl[\nabla f_{\pi(i)}(x_{i-1})-\nabla f_{\pi(i)}(x)\bigr],
 \qquad
 V:=\frac1n\sum_{i=1}^n\|\nabla f_i(x)-g\|^2.
\label{eq:prox-fixed-epoch}
\end{equation}
The vector \(R\) is the smooth within-epoch error analyzed in Sections~4.4--4.6.
The next lemma combines those bounds with the proximal fixed-point comparison.

\begin{lemma}[One-epoch recursion for composite ProxRR]
\label{lem:prox-one-epoch}
Under \eqref{eq:prox-stepsize},
\begin{align}
 \E\bigl[\norm{x^+-x^\dagger}^2\mid x\bigr]
 \le{}&
 \left(1-\frac{\mu h}{2}\right)\norm e^2
 +9\frac{L^2}{\mu}h^3\beta_\star^2
 +32\frac{L^2V}{\mu}n\eta^3                                  \notag\\
 &+3L^2Vn^3\eta^4
 +\frac{16c_\nu^2\rho_\nu^2}{\mu}
   n^{2+\nu}\eta^{3+2\nu}V^{1+\nu}.
\label{eq:prox-one-epoch}
\end{align}
\end{lemma}
\paragraph{Interpretation.}
The recursion contains the same reshuffling terms as the H\"older one-epoch recursion~\eqref{eq:holder-sc-one-epoch}, together with the additional term \(h^3\beta_\star^2\).
The extra term appears because \(\nabla F(x^\dagger)\) need not vanish: the nonsmooth regularizer balances it at the composite optimum.
See Appendix~\ref{app:proximal-proofs}.


\subsection{Upper bounds and objective certification}

\begin{theorem}[Composite ProxRR finite-horizon bound]
\label{thm:prox-global}
Under the assumptions of this section and \eqref{eq:prox-stepsize}, for every \(K\ge1\),
\begin{align}
 \E\norm{y_K-x^\dagger}^2
 \le{}&
 e^{-\mu\eta T/2}D_{\mathrm c}^2
 +\frac{18L^2}{\mu^2}h^2\beta_\star^2
 +\frac{64L^2G_{\mathrm c}^2}{\mu^2}\eta^2                    \notag\\
 &+\frac{6L^2G_{\mathrm c}^2}{\mu}n^2\eta^3
 +\frac{32c_\nu^2\rho_\nu^2G_{\mathrm c}^{2+2\nu}}{\mu^2}
   n^{1+\nu}\eta^{2+2\nu}.
\label{eq:prox-global}
\end{align}
\end{theorem}
\paragraph{Interpretation.}
Trajectory confinement makes the one-epoch constants uniform.
Applying Lemma \ref{lem:prox-one-epoch} conditionally at each epoch and summing the resulting affine contraction yields the theorem.
The first new term is the accumulated epoch-end splitting error; the remaining terms come from current-label dependence, the use of changing gradient-evaluation points within the epoch, and the Taylor remainder, exactly as in the smooth problem.
See Appendix~\ref{app:proximal-proofs}.

\begin{corollary}[Horizon-tuned composite rate]
\label{cor:prox-rate}
Let \(T=nK\ge3\), choose
\begin{equation}
 \eta=\frac{4\log T}{\mu T},
\label{eq:prox-horizon-eta}
\end{equation}
and suppose
\begin{equation}
 K\ge256\frac{L^2}{\mu^2}\log T.
\label{eq:prox-epoch-condition}
\end{equation}
Then
\begin{equation}
 \E\norm{y_K-x^\dagger}^2
 =\wtO\!\left(
 \frac{\beta_\star^2}{K^2}
 +T^{-2}+n^2T^{-3}+n^{1+\nu}T^{-2-2\nu}
 \right).
\label{eq:prox-rate-general}
\end{equation}
If \(\nu\ge\tfrac12\), this reduces to
\begin{equation}
 \E\norm{y_K-x^\dagger}^2
 =\wtO\!\left(
 \frac{\beta_\star^2}{K^2}
 +\frac1{(nK)^2}+\frac1{nK^3}
 \right).
\label{eq:prox-rate-sharp}
\end{equation}
\end{corollary}
\paragraph{Interpretation and proof idea.}
With \(\eta\asymp\log(T)/T\), the splitting term becomes \(\beta_\star^2/K^2\).
When \(\nu\ge\tfrac12\), the H\"older-curvature term is higher order than \(n^2T^{-3}\), leaving the three displayed leading terms.
See Appendix~\ref{app:proximal-proofs}.

The theorem is stated in squared distance because an extended-valued regularizer need not admit a local upper bound of the raw objective gap by \(\norm{y-x^\dagger}^2\).
One deterministic proximal-gradient step provides such a certificate.

\begin{corollary}[Certified composite objective]
\label{cor:prox-objective-certificate}
Let
\begin{equation}
 \widehat y_K
 :=\operatorname{prox}_{\psi/L}
 \left(y_K-\frac1L\nabla F(y_K)\right).
\label{eq:prox-certified-output}
\end{equation}
Then
\begin{equation}
 \mathcal P(\widehat y_K)-\mathcal P(x^\dagger)
 \le\frac L2\norm{y_K-x^\dagger}^2.
\label{eq:prox-objective-certificate}
\end{equation}
Thus the bounds of \cref{thm:prox-global} and Corollary \ref{cor:prox-rate} also hold for the expected certified objective gap after multiplication by \(L/2\).
\end{corollary}
\paragraph{Interpretation.}
For a general extended-valued regularizer, squared distance does not directly upper-bound the raw objective gap at \(y_K\).
One full proximal-gradient step minimizes the standard quadratic upper model of \(F+\psi\), which gives the displayed certificate.
See Appendix~\ref{app:proximal-proofs} for the proof.

\subsection{A matching proximal-splitting lower bound}

The term \(\beta_\star^2/K^2\) is not a permutation effect.
It remains when all components are identical, because an epoch of \(n\) gradient steps is not exactly one gradient step of size \(h=n\eta\) before the proximal map.

\begin{lemma}[Sequential-step ratio]
\label{lem:prox-ratio}
Let \(n\ge2\), \(a\in(0,1/2]\), and \(q=(1-a)^n\).
Then
\begin{equation}
 \frac{na}{1-q}-1\ge\frac18\min\{na,1\}.
\label{eq:prox-ratio}
\end{equation}
\end{lemma}
\paragraph{Interpretation.}
The ratio measures the discrepancy between \(n\) sequential gradient steps of size \(\eta\) and one gradient step of size \(n\eta\).
The lemma gives a uniform lower bound on that discrepancy in both regimes \(n\mu\eta\le1\) and \(n\mu\eta>1\).
It follows from elementary bounds on a geometric series; see Appendix~\ref{app:proximal-proofs} for the proof.

\begin{theorem}[Unavoidable epoch-end splitting term]
\label{thm:prox-lower}
Fix \(n\ge2\), \(K\ge1\), \(\mu>0\), and \(\lambda\in(0,\mu]\).
There is a two-dimensional composite finite sum with \(L=\mu\), constant average Hessian, zero centered component variance, and \(\beta_\star=\lambda\), such that every constant stepsize satisfying \(0<\mu\eta\le1/2\) obeys
\begin{equation}
 \norm{y_K-x^\dagger}^2
 \ge c_{\mathrm{prox}}\frac{\lambda^2}{\mu^2K^2},
\label{eq:prox-lower-distance}
\end{equation}
for a universal \(c_{\mathrm{prox}}>0\).
On the same instance,
\begin{equation}
 \mathcal P(y_K)-\mathcal P(x^\dagger)
 =\frac\mu2\norm{y_K-x^\dagger}^2.
\label{eq:prox-lower-objective-equality}
\end{equation}
\end{theorem}
\paragraph{Interpretation.}
The construction has identical components, so there is no randomness and no centered permutation variance.
Nevertheless, applying the proximal map only at the end of an epoch creates an error of order \(1/K^2\).
Hence this term cannot be removed by a sharper reshuffling analysis.
See Appendix~\ref{app:proximal-proofs} for the proof.


\begin{corollary}[Combined composite lower bound]
\label{cor:prox-combined-lower}
Consider constant stepsizes in the stable range \(0<\mu\eta\le1/2\) and the raw last epoch iterate.
Let \(G_{\rm rr}\) and \(\lambda_{\rm rr}\) denote, respectively, the gradient-bound and strong-convexity parameters of the Safran--Shamir reshuffling block used in the construction, with \(G_{\rm rr}\ge6\lambda_{\rm rr}>0\).
Suppose
\begin{equation}
 \frac{G_{\rm rr}^2}{\lambda_{\rm rr}}
 \left(
 \frac1{(nK)^2}+\frac1{nK^3}
 \right)
 \le \lambda_{\rm rr}.
\label{eq:prox-large-horizon-regime}
\end{equation}
Then the class contains composite instances with a Lipschitz-continuous average Hessian such that
\begin{equation}
 \E\norm{y_K-x^\dagger}^2
 =\Omega\!\left(
 \frac{\beta_\star^2}{K^2}
 +\frac1{(nK)^2}+\frac1{nK^3}
 \right),
\label{eq:prox-combined-lower}
\end{equation}
up to fixed condition and heterogeneity parameters.
The same order can be obtained for the raw composite objective gap.
Thus \eqref{eq:prox-rate-sharp} is tight in its \((n,K)\)-dependence, up to logarithms, for squared distance when \(\nu\ge1/2\).
\end{corollary}
\paragraph{Interpretation.}
A Cartesian product combines the deterministic proximal-splitting instance with the classical quadratic reshuffling hard instance.
The algorithm and objective remain block-separable, so the squared distances and objective gaps add.
This matches all three leading terms in Corollary \ref{cor:prox-rate} up to logarithms.
See Appendix~\ref{app:proximal-proofs} for the proof.

The upper and lower bounds separate two effects.
The \(\beta_\star^2/K^2\) term is an epoch-level proximal-splitting error; it vanishes when \(g^\dagger=0\) but cannot be improved by better permutation moments.
The other two leading terms are genuine reshuffling errors and are improved by centering.
Relative to \cite{mishchenko2022,liu2024}, the upper theorem permits nonconvex components and needs H\"older continuity only of the average Hessian.
Its direct objective statement is for the certified output \eqref{eq:prox-certified-output}.
A raw last-iterate objective guarantee under the same weak component assumptions is not established here.

\section{Scope and limitations}
\label{sec:discussion}
\paragraph{Algorithms and outputs.}
The finite-horizon results in Theorems~\ref{thm:lh-main}, \ref{thm:smooth-main}, and~\ref{thm:holder-sc} use fresh independent reshuffling, a constant component stepsize, and the last epoch iterate \(y_K\).
The all-epoch result in Theorem~\ref{thm:all-k-holder} instead uses the decreasing schedule \eqref{eq:all-k-first-eta}--\eqref{eq:all-k-later-eta}, convex components, and the bounded-iterates condition in Assumption~\ref{ass:all-k-bounded}.
The composite result in Theorem~\ref{thm:prox-global} concerns the epoch-wise ProxRR method \eqref{eq:proxrr-inner}--\eqref{eq:proxrr-prox} and directly controls squared distance to \(x^\dagger\).

Corollary~\ref{cor:prox-objective-certificate} does not claim a direct bound on \(\mathcal P(y_K)-\mathcal P(x^\dagger)\).
The regularizer \(\psi\) may take the value \(+\infty\), so a small distance \(\|y_K-x^\dagger\|\) alone need not upper-bound the raw objective gap.
Instead, the corollary applies one full proximal-gradient step to \(y_K\), producing \(\widehat y_K\), and proves
\[
 \mathcal P(\widehat y_K)-\mathcal P(x^\dagger)
 \le \frac L2\|y_K-x^\dagger\|^2.
\]

\paragraph{Lower-bound scope.}
The sharpness claim in Corollary~\ref{cor:holder-sc-rate} compares its constant-stepsize, last-iterate upper bound with the quadratic lower bound \eqref{eq:intro-quadratic-lb} of Safran and Shamir~\cite{safran2020}.
It does not cover varying stepsizes, adaptive algorithms, or averaged outputs.
Theorem~\ref{thm:prox-lower} proves that the \(\beta_\star^2/K^2\) ProxRR splitting term is unavoidable even when all components are identical.
Corollary~\ref{cor:prox-combined-lower} combines that construction with the Safran--Shamir quadratic instance and matches the three leading terms in Corollary~\ref{cor:prox-rate}, up to logarithms and fixed problem parameters, under the constant-stepsize and large-horizon conditions stated there.

\paragraph{Epoch and parameter dependence.}
Corollaries~\ref{cor:lh-rate} and~\ref{cor:holder-sc-rate} require \(K\ge128(L/\mu)^2\log T\), while Corollary~\ref{cor:prox-rate} uses the slightly stronger requirement \(K\ge256(L/\mu)^2\log T\).
By contrast, Theorem~\ref{thm:all-k-holder} is valid for every \(K\ge1\), at the price of component convexity, varying stepsizes, and Assumption~\ref{ass:all-k-bounded}.
The displayed constants in these results are conservative in the condition number and in the problem-dependent gradient bounds; the theorems primarily target the dependence on \((n,K)\).

\paragraph{Global assumptions and probability mode.}
Theorems~\ref{thm:lh-main} and~\ref{thm:holder-sc} assume global component smoothness, global strong convexity of the average, and global regularity of the average Hessian.
Theorem~\ref{thm:prox-global} makes the corresponding assumptions for the smooth part of the composite objective.
The deterministic confinement results in Lemma~\ref{lem:confinement} and \ref{lem:prox-confinement} suggest that local assumptions on a sufficiently large invariant region may suffice.
All results are in expectation.
A high-probability extension would also need concentration for the dependent label and evaluation-point errors in Lemma~\ref{lem:coupling} and \ref{lem:B}.

\paragraph{Optimality status.}
For \(\nu\ge1/2\), Corollary~\ref{cor:holder-sc-rate} is sharp in its \((n,K)\)-dependence for constant stepsizes and the raw last epoch iterate, because its class contains the quadratic instances underlying \eqref{eq:intro-quadratic-lb}.
Corollary~\ref{cor:all-k-holder-rate} reaches the same scale with a varying schedule, but the cited lower bound does not apply to arbitrary varying schedules.
For \(\nu<1/2\), the upper bound in Corollary~\ref{cor:holder-sc-rate} contains the additional term \(n^{1+\nu}T^{-2-2\nu}\); this work does not prove that the term is necessary.
For ProxRR, Theorem~\ref{thm:prox-lower} establishes the splitting lower bound for all constant stepsizes in its stated range, while Corollary~\ref{cor:prox-combined-lower} gives the combined lower bound under the additional large-horizon condition \eqref{eq:prox-large-horizon-regime}.

\section{Conclusion}
In this paper, we analyze SGD with random reshuffling when the Hessian of the average objective is Lipschitz or H\"older continuous.
The key technical idea is to decompose each permutation prefix into a deterministic mean part and a centered finite-population fluctuation.
We expand the average gradient only after making this decomposition, so the linear term in the centered fluctuation vanishes in expectation.

Two additional errors must still be controlled.
First, the component used at a given position depends on the labels already used in that epoch; the label-transposition argument in Lemma~\ref{lem:coupling} controls this effect.
Second, the actual algorithm evaluates gradients at changing inner iterates, whereas the reference prefixes use gradients evaluated at the epoch start; Lemma~\ref{lem:B} controls the resulting evaluation-point error.

For a smooth strongly convex average with a Lipschitz Hessian, this yields the sharp last-iterate rate
\[
 \wtO\!\left((nK)^{-2}+(nK^3)^{-1}\right).
\]
A \(\nu\)-H\"older average Hessian contributes only \(\wtO(n^{1+\nu}T^{-2-2\nu})\), so \(\nu\ge1/2\) preserves the quadratic rate.
Under convex components and the bounded-iterates condition, the varying-stepsize analysis in Theorem~\ref{thm:all-k-holder} removes the lower bound on the number of epochs and reaches the same scale once the total number of component updates is comparable to the condition-number scale.

For epoch-wise ProxRR, nonexpansiveness reduces the post-proximal distance analysis to the same pre-proximal smooth epoch error.
The resulting rate contains an additional \(\beta_\star^2/K^2\) splitting term, and Theorem~\ref{thm:prox-lower} shows that this term is intrinsic to applying the common proximal map only once per epoch.
These results isolate finite-population centering and epoch-end splitting as distinct mechanisms that govern sharp reshuffling rates.

Natural directions for future work include H\"older-specific lower bounds for \(\nu<1/2\), high-probability guarantees, and all-epoch results that avoid a bounded-iterates assumption.

\section*{Statement on AI Usage}
Large language models (LLMs), including OpenAI's GPT-5.5 Pro and GPT-5.6 Pro, were used as interactive research assistants to explore proof strategies, check technical derivations, and improve the exposition.
The author reviewed the final mathematical statements and proofs and takes full responsibility for the content and any remaining errors.

\appendix

\section{Proofs for the baseline strongly convex results}
\label{app:baseline-proofs}

\subsection{Common deterministic and finite-population estimates}\label{app:baseline-proofs-shared}

\begin{proof}[\textbf{Proof of Lemma \ref{lem:basic-sc-smooth}}]
Averaging the component smoothness inequalities shows that \(\nabla F\) is \(L\)-Lipschitz; \eqref{eq:g-L-e} follows from \(\nabla F(x_\star)=0\).
For \(L>\mu\), apply co-coercivity to the convex, \((L-\mu)\)-smooth function \(G(z)=F(z)-\frac\mu2\norm z^2\).
With \(y=x_\star\), the inequality
\[
\ip{x-y}{\nabla G(x)-\nabla G(y)}
\ge \frac{1}{L-\mu}\norm{\nabla G(x)-\nabla G(y)}^2
\]
rearranges to the first line of \eqref{eq:interpolation}; the case \(L=\mu\) follows by continuity.
The second line uses \(0<\mu\le L\).
Finally, strong convexity gives the lower bound in \eqref{eq:function-distance}, and smoothness applied at \(x_\star\) gives its upper bound.
\end{proof}

\begin{proof}[\textbf{Proof of Lemma \ref{lem:confinement}}]
Fix an epoch start \(x\) with \(\norm{x-x_\star}\le D\).
Suppose that \(i\) is the first inner index leaving the ball of radius \(2D\).
Before that index, component smoothness and \eqref{eq:DG} imply \(\norm{\nabla f_j(x_q)}\le G\) for every component \(j\).
Hence \(\norm{x_i-x}\le i\eta G\le hG\le D/8\) under \eqref{eq:stepsize}, contradicting \(\norm{x_i-x_\star}>2D\).
The same estimate proves the gradient bound in \eqref{eq:inner-bound}.

For the epoch endpoint, let \(e=x-x_\star\) and \(g=\nabla F(x)\).
Adding and subtracting gradients at the epoch start gives
\begin{equation}
x_n-x_\star
=e-hg
+\eta\sum_{i=1}^n
\bigl[\nabla f_{\pi(i)}(x)-\nabla f_{\pi(i)}(x_{i-1})\bigr].
\label{eq:confinement-endpoint}
\end{equation}
Strong convexity and \eqref{eq:g-L-e} imply
\begin{equation}
\norm{e-hg}^2\le(1-\mu h)\norm e^2,
\qquad
\norm{e-hg}
\le\left(1-\frac{\mu h}{2}\right)\norm e,
\label{eq:gd-contraction-norm}
\end{equation}
where the second inequality uses \(\mu h\le1\).
The inner bounds also give
\begin{equation}
\eta\sum_{i=1}^n
\norm{\nabla f_{\pi(i)}(x)-\nabla f_{\pi(i)}(x_{i-1})}
\le\frac{h^2LG}{2}
\le2h^2L^2D.
\label{eq:confinement-error}
\end{equation}
Combining \eqref{eq:gd-contraction-norm} and \eqref{eq:confinement-error} gives
\[
 \|x_n-x_\star\|
 \le \left(1-\frac{\mu h}{2}\right)\|e\|+2h^2L^2D
 \le \left(1-\frac{\mu h}{2}+2h^2L^2\right)D
 \le D,
\]
where the last inequality follows from \eqref{eq:stepsize}.
Induction over epochs completes the proof.
\end{proof}

\begin{proof}[\textbf{Proof of Lemma \ref{lem:prefix}}]
Let \(z_i=g_i-g\), so \(\sum_i z_i=0\) and \(Z_m=\sum_{j=1}^m z_{\pi(j)}\).
If \(I_i\) indicates whether index \(i\) appears in the first \(m\) positions, then
\[
\E I_i=\frac mn,
\qquad
\E[I_iI_j]=\frac{m(m-1)}{n(n-1)}\quad(i\ne j).
\]
Using \(\sum_{i\ne j}\ip{z_i}{z_j}=-\sum_i\norm{z_i}^2\),
\begin{align*}
\E\norm{Z_m}^2
&=\frac mn\sum_i\norm{z_i}^2
 +\frac{m(m-1)}{n(n-1)}\sum_{i\ne j}\ip{z_i}{z_j}\\
&=\frac{m(n-m)}{n-1}\frac1n\sum_i\norm{z_i}^2.
\end{align*}
The cross term between \(mg\) and \(Z_m\) has zero expectation, proving \eqref{eq:prefix-exact}.
Summing \(\sum_{m=0}^{n-1}m^2=n(n-1)(2n-1)/6\) and \(\sum_{m=0}^{n-1}m(n-m)/(n-1)=n(n+1)/6\) gives the remaining identities.
\end{proof}

\begin{proof}[\textbf{Proof of Lemma \ref{lem:path-energy}}]
Within this proof, set \(\mathcal E_m:=\sum_{q=0}^{m-1}\E\norm{d_q}^2\), with \(\mathcal E_0=0\).
Cauchy--Schwarz and component smoothness give
\begin{equation}
\E\norm{R_m}^2
\le m\sum_{j=1}^m
\E\norm{\nabla f_{\pi(j)}(x_{j-1})-g_{\pi(j)}}^2
\le mL^2\mathcal E_m.
\label{eq:partial-error-bound}
\end{equation}
Since \(\mathcal E_{m+1}=\mathcal E_m+\E\norm{d_m}^2\), the exact path identity \eqref{eq:path-exact} and \(\norm{u+v}^2\le2\norm u^2+2\norm v^2\) imply
\begin{align}
\mathcal E_{m+1}
&\le \mathcal E_m
 +2\eta^2\E\norm{mg+Z_m}^2
 +2\eta^2\E\norm{R_m}^2                                      \notag\\
&\le(1+2\eta^2mL^2)\mathcal E_m
 +2\eta^2\E\norm{mg+Z_m}^2.
\label{eq:path-energy-recursion}
\end{align}
No independence between \(Z_m\) and \(R_m\) is used: their cross term is controlled by the deterministic two-vector inequality.

Iterating \eqref{eq:path-energy-recursion} from \(\mathcal E_0=0\) gives
\[
\mathcal E_n
\le2\eta^2\sum_{j=0}^{n-1}
\E\norm{jg+Z_j}^2
\prod_{\ell=j+1}^{n-1}(1+2\eta^2\ell L^2).
\]
The product is at most \(\exp(2\eta^2L^2\sum_{\ell=0}^{n-1}\ell) \le e^{h^2L^2}\).
Applying \eqref{eq:prefix-sum} yields
\[
\mathcal E_n
\le\frac23e^{h^2L^2}\eta^2
\bigl(n^3\norm g^2+n^2V\bigr).
\]
The prefactor is below one under \eqref{eq:path-stability}, proving \eqref{eq:path-energy-bound}.
Finally, \eqref{eq:partial-error-bound} at \(m=n\) gives \(\E\norm{R_n}^2\le nL^2\mathcal E_n\), which is \eqref{eq:R-second-bound}.
\end{proof}

\subsection{Merely smooth regime}\label{app:baseline-proofs-smooth}

\begin{proof}[\textbf{Proof of Lemma \ref{lem:smooth-mean}}]
Jensen's inequality and Lemma \ref{lem:path-energy} yield
\[
\norm{\E R}
\le\sqrt{\E\norm{R}^2}
\le L\eta\sqrt{n^4\norm{g}^2+n^3V}.
\]
Apply \(\sqrt{a+b}\le\sqrt a+\sqrt b\).
\end{proof}

\begin{proof}[\textbf{Proof of Lemma \ref{lem:smooth-one-epoch}}]
From \eqref{eq:endpoint-R}, \(e^+:=x_n-x_\star=e-hg-\eta R\).
Expanding exactly,
\begin{align}
\E\norm{e^+}^2
={}&\norm{e}^2-2h\ip{e}{g}+h^2\norm{g}^2
+\eta^2\E\norm{R}^2\notag\\
&-2\eta\ip{e}{\E R}
+2\eta h\ip{g}{\E R}.
\label{eq:exact-endpoint-expansion}
\end{align}
By \eqref{eq:interpolation} and
\begin{equation}
2\eta h\ip{g}{\E R}
\le h^2\norm{g}^2+\eta^2\norm{\E R}^2
\le h^2\norm{g}^2+\eta^2\E\norm{R}^2,
\label{eq:gR-young}
\end{equation}
we obtain
\begin{align}
\E\norm{e^+}^2
\le{}&(1-\mu h)\norm{e}^2
-h\left(\frac1L-2h\right)\norm{g}^2
+2\eta^2\E\norm{R}^2
+2\eta\norm{e}\norm{\E R}.
\label{eq:smooth-master}
\end{align}
The mean term is bounded using Lemma \ref{lem:smooth-mean}:
\begin{align}
2\eta\norm{e}\norm{\E R}
&\le2Lh^2\norm{e}\norm{g}
+2L\eta^2n^{3/2}\sqrt V\norm{e}\notag\\
&\le2L^2h^2\norm{e}^2
+\frac{\mu h}{8}\norm{e}^2
+\frac{8L^2}{\mu}n^2\eta^3V.
\label{eq:smooth-mean-cross}
\end{align}
For the second inequality we used \eqref{eq:g-L-e} and Young's inequality \(ca\le\lambda a^2+c^2/(4\lambda)\) with \(\lambda=\mu h/8\) for the variance term.
Also, Lemma \ref{lem:path-energy} gives
\begin{equation}
2\eta^2\E\norm{R}^2
\le2L^2h^4\norm{g}^2+2L^2n^3\eta^4V.
\label{eq:smooth-R-second-in-recursion}
\end{equation}

It remains to absorb the state-dependent terms.
Set
\begin{equation}
s:=hL,
\qquad
r:=\frac{\mu}{L}.
\label{eq:sr}
\end{equation}
Under \eqref{eq:stepsize}, \(s\le r/32\).
After division by \(h/L\), the coefficient of \(\norm g^2\) is
\begin{equation}
-(1-2s)+2s^3<0,
\label{eq:smooth-g-absorb}
\end{equation}
so it can be discarded.
The remaining positive state correction is at most \(3\mu h\norm e^2/16\), which is smaller than half of the contraction margin.
Substitution into \eqref{eq:smooth-master} proves \eqref{eq:smooth-one-epoch}.
\end{proof}

\begin{proof}[\textbf{Proof of \cref{thm:smooth-main}}]
By Lemma \ref{lem:confinement}, every epoch starts inside the radius-\(D\) ball.
Conditional on \(y_k\), Lemma \ref{lem:smooth-one-epoch} applies to the fresh independent permutation.
Since \(V\le G^2\), tower expectation yields
\begin{align}
a_{k+1}
\le{}&\left(1-\frac{\mu h}{2}\right)a_k
+\frac{8L^2G^2}{\mu}n^2\eta^3
+2L^2G^2n^3\eta^4,
\label{eq:smooth-global-recursion}
\end{align}
where \(a_k:=\E\norm{y_k-x_\star}^2\).
Let \(\theta=1-\mu h/2\).
Iterating \eqref{eq:smooth-global-recursion} gives
\begin{equation}
a_K\le\theta^KD^2+\frac{1-\theta^K}{1-\theta}b,
\label{eq:geometric-generic}
\end{equation}
where \(b\) is its additive term.
The elementary bounds
\begin{equation}
\theta^K\le\exp\!\left(-\frac{\mu\eta T}{2}\right),
\qquad
\frac{1-\theta^K}{1-\theta}\le\frac{2}{\mu n\eta}
\label{eq:geometric-bounds}
\end{equation}
then give \eqref{eq:smooth-global}.
\end{proof}

\begin{proof}[\textbf{Proof of Corollary \ref{cor:smooth-rate}}]
The epoch condition implies \eqref{eq:stepsize}, and the horizon choice makes the initialization term \(D^2/T^2\).
Substitution into \eqref{eq:smooth-global} gives \eqref{eq:smooth-cor-explicit}.
Since \(n^2/T^3\le n/T^2\), the stated rate follows; use \eqref{eq:function-distance} for objective error.
\end{proof}

\subsection{Lipschitz-Hessian regime}\label{app:baseline-proofs-Lip-Hessian}

\begin{proof}[\textbf{Proof of Lemma \ref{lem:B}}]
The actual iterate uses the gradients \(\nabla f_{\pi(j)}(x_{j-1})\) for \(j<i\), whereas \(\widetilde x_{i-1}\) uses the same component labels but evaluates those gradients at \(x\).
Subtracting their definitions and using the definition of \(R_{i-1}\) gives
\[
x_{i-1}-\widetilde x_{i-1}=-\eta R_{i-1}.
\]
Component smoothness therefore implies
\[
\norm{R^{\mathrm{eval}}}
\le\eta L\sum_{m=0}^{n-1}\norm{R_m}.
\]
Cauchy--Schwarz, \eqref{eq:partial-error-bound}, and monotonicity of \(\mathcal E_m:=\sum_{q=0}^{m-1}\E\norm{d_q}^2\) give
\begin{align*}
\E\norm{R^{\mathrm{eval}}}^2
&\le\eta^2L^2n\sum_{m=0}^{n-1}\E\norm{R_m}^2\\
&\le\eta^2L^4n\sum_{m=0}^{n-1}m\mathcal E_m
\le\frac12\eta^2L^4n^3\mathcal E_n.
\end{align*}
Substituting the bound \eqref{eq:path-energy-bound} for \(\mathcal E_n\) proves \eqref{eq:B-bound}.
\end{proof}

\begin{proof}[\textbf{Proof of Lemma \ref{lem:coupling}}]
Fix the position \(i\).
For \(r\in[n]\), let \(\mathcal P_r\) be the uniform distribution on permutations conditioned on \(\pi(i)=r\), and let \(\E_r\) denote expectation under \(\mathcal P_r\).
Write
\[
 W_\pi:=x-\eta\sum_{j=1}^{i-1}g_{\pi(j)}
       =\widetilde x_{i-1}.
\]
Since every label is marginally uniform at position \(i\), conditioning on that label gives
\begin{align}
 \E\nabla f_{\pi(i)}(W_\pi)
 &=\frac1n\sum_{r=1}^n \E_r\nabla f_r(W_\pi),
\label{eq:coupling-first-expectation}\\
 \E\nabla F(W_\pi)
 &=\frac1{n^2}\sum_{r=1}^n\sum_{s=1}^n
     \E_s\nabla f_r(W_\pi).
\label{eq:coupling-second-expectation}
\end{align}
Consequently, with \(\Delta_i:=\E\nabla f_{\pi(i)}(W_\pi)-\E\nabla F(W_\pi)\),
\begin{equation}
 \Delta_i=\frac1{n^2}\sum_{r=1}^n\sum_{s=1}^n
 \left[\E_r\nabla f_r(W_\pi)-\E_s\nabla f_r(W_\pi)\right].
\label{eq:coupling-double-sum}
\end{equation}

Fix \(r\ne s\).
If \(\pi\sim\mathcal P_r\), let \(\tau=T_{r,s}\pi\) be obtained by swapping the labels \(r\) and \(s\) throughout the permutation.
Then \(\tau(i)=s\), and \(T_{r,s}\) is an involutive bijection from the support of \(\mathcal P_r\) to that of \(\mathcal P_s\).
Because both conditional laws are uniform,
\begin{equation}
 \E_s\nabla f_r(W_\pi)=\E_r\nabla f_r(W_{T_{r,s}\pi}).
\label{eq:coupling-change-of-variables}
\end{equation}

It remains to compare the two evaluation points.
Under \(\mathcal P_r\), label \(r\) occupies position \(i\) and hence does not occur among the first \(i-1\) labels.
If label \(s\) also occurs after position \(i\), the first \(i-1\) labels are unchanged by the swap and \(W_\pi=W_{T_{r,s}\pi}\).
If \(s\) occurs before position \(i\), the swap replaces the single vector \(g_s\) in the preceding sum by \(g_r\), so
\[
 W_{T_{r,s}\pi}=W_\pi-\eta(g_r-g_s).
\]
In either case,
\begin{equation}
 \|W_\pi-W_{T_{r,s}\pi}\|\le\eta\|g_r-g_s\|.
\label{eq:coupling-point-distance}
\end{equation}
Using \eqref{eq:coupling-change-of-variables}, Jensen's inequality, and the \(L\)-Lipschitz continuity of \(\nabla f_r\), we obtain
\begin{equation}
 \left\|\E_r\nabla f_r(W_\pi)-\E_s\nabla f_r(W_\pi)\right\|
 \le \eta L\|g_r-g_s\|.
\label{eq:coupling-pairwise-bound}
\end{equation}
The same inequality is trivial for \(r=s\).
Applying the triangle inequality to \eqref{eq:coupling-double-sum} and then Cauchy--Schwarz gives
\begin{align*}
 \|\Delta_i\|
 &\le \frac{\eta L}{n^2}\sum_{r,s}\|g_r-g_s\|\\
 &\le \eta L\left(\frac1{n^2}\sum_{r,s}\|g_r-g_s\|^2\right)^{1/2}.
\end{align*}
Finally, if \(z_r=g_r-g\), then \(\sum_r z_r=0\), and hence
\[
 \frac1{n^2}\sum_{r,s}\|g_r-g_s\|^2
 =\frac1{n^2}\sum_{r,s}\|z_r-z_s\|^2
 =\frac2n\sum_{r=1}^n\|z_r\|^2
 =2V.
\]
Substitution proves \eqref{eq:coupling}.
\end{proof}

\begin{proof}[\textbf{Proof of Lemma \ref{lem:mean-A}}]
From the definition of \(R^{\mathrm{ref}}\), the marginal identity \(\E g_{\pi(m+1)}=g\), and the change of index \(m=i-1\),
\begin{equation}
 \E R^{\mathrm{ref}}
 =\sum_{m=0}^{n-1}
  \left[\E\nabla f_{\pi(m+1)}(\widetilde x_m)-g\right].
\label{eq:mean-reference-start}
\end{equation}
For each summand, add and subtract \(\E\nabla F(\widetilde x_m)\) and \(\nabla F(a_m)\).
The definitions in the statement of the lemma then give the exact identity
\[
 \E R^{\mathrm{ref}}=r_{\mathrm{lab}}+r_{\mathrm{curv}}+b.
\]
Thus it remains only to bound the three explicitly defined vectors.

Lemma~\ref{lem:coupling}, applied at position \(m+1\), and the triangle inequality give
\[
 \|r_{\mathrm{lab}}\|
 \le\sum_{m=0}^{n-1}
 \left\|\E\nabla f_{\pi(m+1)}(\widetilde x_m)
       -\E\nabla F(\widetilde x_m)\right\|
 \le n\eta L\sqrt{2V}.
\]

For the curvature term, recall that \(\widetilde x_m=a_m-\eta Z_m\).
The integral Taylor formula gives
\begin{align}
 \nabla F(a_m-\eta Z_m)
 &=\nabla F(a_m)-\eta\nabla^2F(a_m)Z_m+\varepsilon_m,
\label{eq:centered-taylor}\\
 \varepsilon_m
 &:=-\eta\int_0^1
 \left[\nabla^2F(a_m-t\eta Z_m)-\nabla^2F(a_m)\right]Z_m\,dt.
\label{eq:centered-taylor-remainder}
\end{align}
The Lipschitz-Hessian assumption implies
\begin{equation}
 \|\varepsilon_m\|\le\frac{\rho\eta^2}{2}\|Z_m\|^2.
\label{eq:epsilon-bound}
\end{equation}
Conditional on the epoch start, \(a_m\) and \(\nabla^2F(a_m)\) are deterministic, whereas \(\E Z_m=0\).
Therefore
\[
 \E\nabla F(\widetilde x_m)-\nabla F(a_m)=\E\varepsilon_m,
\]
so \(r_{\mathrm{curv}}=\sum_{m=0}^{n-1}\E\varepsilon_m\).
By Jensen's inequality, \eqref{eq:epsilon-bound}, and \eqref{eq:centered-prefix-sum},
\[
 \|r_{\mathrm{curv}}\|
 \le\frac{\rho\eta^2}{2}\sum_{m=0}^{n-1}\E\|Z_m\|^2
 =\frac{\rho\eta^2n(n+1)}{12}V
 \le\frac{\rho\eta^2n^2}{6}V.
\]

Finally, \(a_m=x-\eta m g\) and the \(L\)-smoothness of \(F\) imply
\[
 \|b\|
 \le\sum_{m=0}^{n-1}\|\nabla F(a_m)-\nabla F(x)\|
 \le\eta L\|g\|\sum_{m=0}^{n-1}m
 =\eta L\frac{n(n-1)}2\|g\|.
\]
These three estimates prove the lemma.
\end{proof}

\begin{proof}[\textbf{{Proof of Lemma \ref{lem:lh-one-epoch}}}]
Starting from \eqref{eq:exact-endpoint-expansion}, the same use of \eqref{eq:interpolation} and \eqref{eq:gR-young} gives
\begin{align}
\E\norm{e^+}^2
\le{}&(1-\mu h)\norm e^2
-h\left(\frac1L-2h\right)\norm g^2
+2\eta^2\E\norm R^2
-2\eta\ip e{\E R}.
\label{eq:lh-master}
\end{align}
By \eqref{eq:R-evaluation} and Lemma \ref{lem:mean-A},
\begin{equation}
\E R=b+r_{\mathrm{lab}}+r_{\mathrm{curv}}+\E R^{\mathrm{eval}}.
\label{eq:ER-full-decomp}
\end{equation}
We bound the four terms separately.
First, \eqref{eq:M-bound}, \(n(n-1)\le n^2\), and \eqref{eq:g-L-e} imply
\begin{equation}
2\eta\lvert\ip e b\rvert
\le h^2L^2\norm e^2.
\label{eq:M-cross}
\end{equation}
Second, \eqref{eq:delta-bound} and Young's inequality give
\begin{align}
2\eta\lvert\ip e{r_{\mathrm{lab}}}\rvert
&\le\frac{\mu h}{16}\norm e^2
+32\frac{L^2V}{\mu}n\eta^3.
\label{eq:delta-cross}
\end{align}
Third, \eqref{eq:q-bound} gives
\begin{align}
2\eta\lvert\ip e{r_{\mathrm{curv}}}\rvert
&\le\frac{\mu h}{16}\norm e^2
+\frac{4\rho^2}{9\mu}n^3\eta^5V^2.
\label{eq:q-cross}
\end{align}
Fourth, Jensen's inequality, Young's inequality, and Lemma \ref{lem:B} yield
\begin{align}
2\eta\lvert\ip e{\E R^{\mathrm{eval}}}\rvert
&\le\frac{\mu h}{16}\norm e^2
+\frac{16\eta}{n\mu}\E\norm{R^{\mathrm{eval}}}^2\notag\\
&\le\frac{\mu h}{16}\norm e^2
+\frac{8L^4h^5}{\mu}\norm g^2
+\frac{8L^4}{\mu}n^4\eta^5V.
\label{eq:B-cross}
\end{align}
Finally, \eqref{eq:R-second-bound} gives
\begin{equation}
2\eta^2\E\norm R^2
\le2L^2h^4\norm g^2+2L^2n^3\eta^4V.
\label{eq:lh-R-second-in-recursion}
\end{equation}

Let \(s=hL\) and \(r=\mu/L\).
Under \eqref{eq:stepsize}, the coefficient of \(\norm g^2\), normalized by \(h/L\), is bounded by
\[
-(1-2s)+2s^3+\frac{8s^4}{r}<0.
\]
The positive distance corrections total at most \(7\mu h\norm e^2/32\), leaving the contraction factor no larger than \(1-\mu h/2\).
The remaining term $8L^4n^4\eta^5V/\mu$ in \eqref{eq:B-cross}, which comes from $R^{\mathrm{eval}}$, is at most one quarter of $L^2n^3\eta^4V$ under \eqref{eq:stepsize}.
Together with \eqref{eq:lh-R-second-in-recursion}, it is bounded by the coefficient \(3\) in \eqref{eq:lh-one-epoch}.
\end{proof}

\begin{proof}[\textbf{{Proof of \cref{thm:lh-main}}}]
Condition on each epoch start and apply Lemma \ref{lem:lh-one-epoch}.
By Lemma \ref{lem:confinement}, the conditioning event holds deterministically, and \(V\le G^2\).
Thus
\begin{align}
a_{k+1}
\le{}&\left(1-\frac{\mu h}{2}\right)a_k
+32\frac{L^2G^2}{\mu}n\eta^3
+3L^2G^2n^3\eta^4
+\frac{4\rho^2G^4}{9\mu}n^3\eta^5,
\label{eq:lh-global-recursion}
\end{align}
where \(a_k:=\E\norm{y_k-x_\star}^2\).
Applying the geometric bounds \eqref{eq:geometric-generic}--\eqref{eq:geometric-bounds} proves \eqref{eq:lh-global}.
\end{proof}

\begin{proof}[\textbf{Proof of Corollary \ref{cor:lh-rate}}]
The epoch condition implies \eqref{eq:stepsize}.
Substitution of \eqref{eq:horizon-eta} into \eqref{eq:lh-global} gives \eqref{eq:lh-cor-explicit}; the curvature term is higher order than \(n^2T^{-3}\).
Apply \eqref{eq:function-distance}.
\end{proof}

\section{Proofs for the H\"older and all-epoch extensions}
\label{app:extension-proofs}

\subsection{H\"older curvature and strong convexity}

\begin{proof}[\textbf{{Proof of Lemma \ref{lem:holder-curvature}}}]
For every \(m\), the integral Taylor formula gives
\begin{align}
\nabla F(a_m-\eta Z_m)
={}&\nabla F(a_m)-\eta\nabla^2F(a_m)Z_m+\varepsilon_m^{(\nu)},
\label{eq:holder-taylor}\\
\varepsilon_m^{(\nu)}
={}&-\eta\int_0^1
 \bigl[\nabla^2F(a_m-t\eta Z_m)-\nabla^2F(a_m)\bigr]Z_m\,dt.
\end{align}
Hence
\begin{equation}
 \norm{\varepsilon_m^{(\nu)}}
 \le \frac{\rho_\nu\eta^{1+\nu}}{1+\nu}
       \norm{Z_m}^{1+\nu}.
\label{eq:holder-remainder}
\end{equation}
The point \(a_m\) and its Hessian are deterministic after conditioning on the epoch start, while \(\E Z_m=0\).
Therefore the linear Taylor term cancels.
Let \(p=(1+\nu)/2\in(1/2,1]\).
Lyapunov's inequality and concavity of \(t\mapsto t^p\) imply
\begin{align}
\sum_{m=0}^{n-1}\E\norm{Z_m}^{1+\nu}
&\le \sum_{m=0}^{n-1}
       \left(\E\norm{Z_m}^2\right)^p                                      \notag\\
&\le n^{1-p}
   \left(\sum_{m=0}^{n-1}\E\norm{Z_m}^2\right)^p                           \notag\\
&=n^{1-p}\left(\frac{n(n+1)}6V\right)^p                                   \notag\\
&\le 2^{-(1+\nu)}n^{(3+\nu)/2}V^{(1+\nu)/2},
\label{eq:holder-prefix-moment}
\end{align}
where the last step uses \(n+1\le 3n/2\) for \(n\ge2\).
Summing \eqref{eq:holder-remainder} proves the result.
\end{proof}

\begin{proof}[\textbf{{Proof of \cref{thm:holder-sc}}}]
The bounds for the dependence of the current label and for the difference between the actual and reference evaluation points are unchanged.
In the decomposition of the conditional mean in Lemma \ref{lem:mean-A}, replace the Lipschitz-Hessian curvature residual by the vector in Lemma \ref{lem:holder-curvature}; denote it by \(r_{\mathrm{curv}}^{(\nu)}\).
Thus
\begin{equation}
 \norm{r_{\mathrm{curv}}^{(\nu)}}
 \le c_\nu\rho_\nu\eta^{1+\nu}
      n^{(3+\nu)/2}V^{(1+\nu)/2}.
\label{eq:holder-curvature-residual}
\end{equation}
In the squared-distance endpoint expansion, its cross term satisfies
\begin{align}
2\eta\abs{\ip{e}{r_{\mathrm{curv}}^{(\nu)}}}
&\le \frac{\mu h}{16}\norm e^2
 +\frac{16c_\nu^2\rho_\nu^2}{\mu}
  n^{2+\nu}\eta^{3+2\nu}V^{1+\nu}.
\label{eq:holder-sc-cross}
\end{align}
Indeed, this is Young's inequality with quadratic coefficient \(\mu h/16\) and \(h=n\eta\).
Every other line of the proof of Lemma \ref{lem:lh-one-epoch} is unchanged.
Consequently,
\begin{align}
\E[\norm{x_n-x_\star}^2\mid x]
\le{}&\left(1-\frac{\mu h}{2}\right)\norm e^2
 +32\frac{L^2V}{\mu}n\eta^3
 +3L^2Vn^3\eta^4                                             \notag\\
&+\frac{16c_\nu^2\rho_\nu^2}{\mu}
 n^{2+\nu}\eta^{3+2\nu}V^{1+\nu}.
\label{eq:holder-sc-one-epoch}
\end{align}
Trajectory confinement gives \(V\le G^2\) at every epoch start.
Geometric summation by at most \(2/(\mu n\eta)\) proves \eqref{eq:holder-sc-global}.
\end{proof}

\begin{proof}[\textbf{{Proof of Corollary \ref{cor:holder-sc-rate}}}]
Substitute \(\eta=4\log T/(\mu T)\) into \eqref{eq:holder-sc-global}.
For the domination statement, note that \(T=nK\ge n\) and
\begin{equation}
 \frac{n^{1+\nu}T^{-2-2\nu}}{n^2T^{-3}}
 =n^{\nu-1}T^{1-2\nu}
 \le n^{-\nu}\le1
 \qquad(\nu\ge\tfrac12).
\end{equation}
The lower-bound comparison with Safran and Shamir~\cite{safran2020} follows because a quadratic average has \(\rho_\nu=0\) for every \(\nu\in(0,1]\).
\end{proof}

\subsection{All-epoch proof under convex components}
\label{app:all-k-holder-proof}

The only external ingredient is the per-update estimate in Lemma \ref{lem:ahn-per-iteration}.
Proposition~D.1 of Ahn, Yun, and Sra~\cite{ahn2020} is stated with a uniform component-gradient bound.
Under Assumption~\ref{ass:all-k-bounded}, every RR iterate lies in the compact set \(\mathcal X\), and hence \eqref{eq:all-k-Gbar} supplies that bound.
The swapped permutation used in their coupling proof is itself a possible RR permutation, so its trajectory is also covered by the same assumption.
Thus no global bounded-gradient assumption is needed.

We first record the summation lemma used below.
It is a version of Chung's lemma adapted to a denominator that increases by \(d\) at each step.

\begin{lemma}[A two-scale polynomial recursion]
\label{lem:two-scale-recursion}
Let \(d>0\), \(q_j=q_0+dj\) with \(q_0\ge d\), and suppose
\begin{equation}
 u_j\le \exp\!\left(-\frac{a d}{q_j}\right)u_{j-1}
       +\frac{B d^r}{q_j^p},
 \qquad j=1,\ldots,N,
\label{eq:two-scale-recursion}
\end{equation}
where \(a>p-1>0\).
Then
\begin{equation}
 u_N\le C_a\left(\frac{q_0}{q_N}\right)^a u_0
       +C_{a,p}\frac{B d^{r-1}}{q_N^{p-1}},
\label{eq:two-scale-recursion-bound}
\end{equation}
where \(C_a\) and \(C_{a,p}\) depend only on \(a\) and \(p\).
\end{lemma}

\paragraph{Interpretation.}
An error of size \(q_j^{-p}\) introduced at index \(j\) is subsequently contracted by approximately \((q_j/q_N)^a\).
Summing those contracted errors over a grid with spacing \(d\) loses one power of \(q_N\) and one power of \(d\), which explains the term \(d^{r-1}q_N^{-(p-1)}\).

\begin{proof}[\textbf{Proof of Lemma \ref{lem:two-scale-recursion}}]
Unrolling \eqref{eq:two-scale-recursion} gives
\begin{align*}
 u_N\le{}&u_0\exp\!\left(-a\sum_{j=1}^N\frac d{q_j}\right)\\
 &+Bd^r\sum_{t=1}^N q_t^{-p}
   \exp\!\left(-a\sum_{j=t+1}^N\frac d{q_j}\right).
\end{align*}
Since \(\log(1+z)\le z\) for \(z\ge0\),
\[
 \sum_{j=t+1}^N\frac d{q_j}
 \ge\sum_{j=t+1}^N\log\frac{q_{j+1}}{q_j}
 =\log\frac{q_{N+1}}{q_{t+1}}.
\]
The condition \(q_t\ge q_0\ge d\) gives \(q_{t+1}\le2q_t\), while \(q_{N+1}\ge q_N\).
Therefore
\[
 \exp\!\left(-a\sum_{j=t+1}^N\frac d{q_j}\right)
 \le 2^a\left(\frac{q_t}{q_N}\right)^a.
\]
The same estimate bounds the initial product, up to a constant depending only on \(a\).
Hence the accumulated additive error is at most a constant multiple of
\[
 Bd^r q_N^{-a}\sum_{t=1}^N q_t^{a-p}.
\]
Because \(a-p>-1\), comparison with the corresponding integral yields \(\sum_{t=1}^Nq_t^{a-p}\le C_{a,p}d^{-1}q_N^{a-p+1}\).
Substitution proves \eqref{eq:two-scale-recursion-bound}.
\end{proof}

\begin{proof}[\textbf{Proof of Theorem \ref{thm:all-k-holder}}]
Let \(\Delta_k:=\E\|y_k-x_\star\|^2\).

\paragraph{Step 1: control the first epoch update by update.}
Since \(q_0=\zeta\kappa\), every stepsize in \eqref{eq:all-k-first-eta} is at most \(2/L\).
Applying Lemma~\ref{lem:ahn-per-iteration} with \(\gamma=2\zeta/[\mu(q_0+i)]\) gives
\[
 u_i\le\left(1-\frac{\zeta}{q_0+i}\right)u_{i-1}
       +\frac{A_2}{(q_0+i)^2}+\frac{A_3}{(q_0+i)^3},
\]
where \(u_i\) is the expected squared distance after \(i\) updates and
\[
 A_2:=\frac{12\zeta^2\bar G^2}{\mu^2},
 \qquad
 A_3:=\frac{32\zeta^3\kappa L\bar G^2}{\mu^3}.
\]
Using \(1-z\le e^{-z}\) and applying Lemma~\ref{lem:two-scale-recursion} with \(d=1\), \(a=\zeta\), and separately \(p=2\) and \(p=3\), we obtain
\begin{equation}
 \Delta_1
 \le C\left(\frac{q_0}{q_1}\right)^{\!\zeta}\Delta_0
      +\frac{C A_2}{q_1}+\frac{C A_3}{q_1^2}.
\label{eq:all-k-first-epoch-bound}
\end{equation}

\paragraph{Step 2: derive one recursion for every later epoch.}
Fix \(k\ge1\), and abbreviate \(q=q_{k+1}\), \(\eta=2\zeta/(\mu q)\), and \(h=n\eta\).
Again \(\eta\le2/L\).

If \(h>\mu/(32L^2)\), apply Lemma \ref{lem:ahn-per-iteration} successively to the \(n\) updates in the epoch.
The multiplicative factors on earlier additive errors are at most one, and \((1-\mu\eta/2)^n\le e^{-\mu h/2}\).
Therefore
\begin{equation}
 \Delta_{k+1}
 \le e^{-\mu h/2}\Delta_k
    +3n\eta^2\bar G^2+4n\eta^3\kappa L\bar G^2.
\label{eq:all-k-early-recursion}
\end{equation}
The inequality \(h>\mu/(32L^2)\) implies
\[
 n\eta^2\le\frac{1024L^4}{\mu^2}n^3\eta^4,
\]
so the \(n\eta^2\) term can be included in an \(n^3\eta^4\) term.

If \(h\le\mu/(32L^2)\), the H\"older one-epoch estimate \eqref{eq:holder-sc-one-epoch} applies at the epoch start \(y_k\).
By \eqref{eq:all-k-Gbar},
\[
 V(y_k)
 \le \frac1n\sum_{i=1}^n\|\nabla f_i(y_k)\|^2
 \le \bar G^2.
\]
Using \(1-\mu h/2\le e^{-\mu h/2}\), this stable regime has the same form as the early regime after constants are enlarged.
Consequently, there are constants \(B_3,B_4,B_\nu\), independent of \(n\) and \(k\), such that both regimes satisfy
\begin{align}
 \Delta_{k+1}
 \le \exp\!\left(-\frac{\zeta n}{q_{k+1}}\right)\Delta_k
       +\frac{B_3 n}{q_{k+1}^3}
       +\frac{B_4 n^3}{q_{k+1}^4} +\frac{B_\nu n^{2+\nu}}{q_{k+1}^{3+2\nu}}.
\label{eq:all-k-uniform-recursion}
\end{align}
Here \(\mu h/2=\zeta n/q_{k+1}\), and all powers of \(2\zeta/\mu\) are absorbed into the constants.

\paragraph{Step 3: sum the later epochs and absorb the first-epoch residual.}
For \(K\ge2\), apply Lemma~\ref{lem:two-scale-recursion} to epochs \(1,\ldots,K-1\), with initial denominator \(q_1\), spacing \(d=n\), \(a=\zeta\), and separately
\[
 (p,r)=(3,1),\qquad(4,3),\qquad(3+2\nu,2+\nu).
\]
These choices are admissible because \(\zeta>4\ge2+2\nu\).
We obtain
\begin{align}
 \Delta_K
 \le{}&C\left(\frac{q_1}{q_K}\right)^{\!\zeta}\Delta_1
       +\frac{C}{q_K^2}
       +\frac{Cn^2}{q_K^3}
       +\frac{Cn^{1+\nu}}{q_K^{2+2\nu}}.
\label{eq:all-k-after-first-epoch}
\end{align}
For \(K=1\), the same final bound follows directly from \eqref{eq:all-k-first-epoch-bound} and the estimates below.

Substitute \eqref{eq:all-k-first-epoch-bound}.
The initialization terms combine to \(C(q_0/q_K)^\zeta\Delta_0\).
It remains to absorb the two first-epoch residuals.
If \(n\le q_0\), then \(q_1\le2q_0\) and
\begin{align*}
 \frac{q_1^{\zeta-1}}{q_K^\zeta}
 &\le \frac{2q_0}{q_K^2},
 &
 \frac{q_1^{\zeta-2}}{q_K^\zeta}
 &\le \frac1{q_K^2}.
\end{align*}
If \(n>q_0\), then \(q_1<2n\); since \(\zeta>4\) and \(q_K\ge q_1\),
\begin{align*}
 \frac{q_1^{\zeta-1}}{q_K^\zeta}
 &\le \frac{q_1^2}{q_K^3}
 \le\frac{4n^2}{q_K^3},
 &
 \frac{q_1^{\zeta-2}}{q_K^\zeta}
 &\le \frac{q_1}{q_K^3}
 \le\frac{2n}{q_K^3}
 \le\frac{n^2}{q_K^3}.
\end{align*}
The fixed factor \(q_0\) and the constants \(A_2,A_3\) can be absorbed into \(C_1,C_2\).
Combining these estimates with \eqref{eq:all-k-after-first-epoch} proves \eqref{eq:all-k-holder-bound}.
\end{proof}

\begin{proof}[\textbf{Proof of Corollary \ref{cor:all-k-holder-rate}}]
Because \(q_K=q_0+nK\ge n\), for \(\nu\ge\tfrac12\),
\[
 \frac{n^{1+\nu}q_K^{-2-2\nu}}{n^2q_K^{-3}}
 =n^{\nu-1}q_K^{1-2\nu}
 \le n^{-\nu}\le1.
\]
If \(T=nK\ge q_0\), then \(T\le q_K\le2T\), and \((q_0/q_K)^\zeta\le q_0^2q_K^{-2}\) because \(\zeta>2\).
Thus every term in \eqref{eq:all-k-holder-sharp} has the order displayed in \eqref{eq:all-k-holder-T-rate}.
The objective bound follows from \eqref{eq:function-distance}.
\end{proof}

\section{Proofs for the composite ProxRR results}
\label{app:proximal-proofs}

\subsection{Proximal upper bounds}

\begin{proof}[\textbf{Proof of Lemma \ref{lem:pairwise-interpolation}}]
For \(L>\mu\), apply co-coercivity to the convex, \((L-\mu)\)-smooth function \(H(z)=F(z)-\frac{\mu}{2}\|z\|^2\) at the pair \((u,v)\):
\[
 \langle u-v,\nabla H(u)-\nabla H(v)\rangle
 \ge \frac{1}{L-\mu}\|\nabla H(u)-\nabla H(v)\|^2.
\]
Expanding \(\nabla H\) and rearranging gives the first inequality in \eqref{eq:pairwise-interpolation}; the case \(L=\mu\) follows by continuity.
The second inequality uses \(0<\mu\le L\).
\end{proof}

\begin{proof}[\textbf{{Proof of Lemma \ref{lem:prox-confinement}}}]
The first-exit argument from Lemma \ref{lem:confinement} applies with \(x_\star,D,G\) replaced by \(x^\dagger,D_{\mathrm c},G_{\mathrm c}\); the stronger restriction \eqref{eq:prox-stepsize} gives \(\norm{x_i-x}\le hG_{\mathrm c}\le D_{\mathrm c}/16\) before any proposed exit.

For the endpoint, write
\begin{equation}
 e:=x-x^\dagger,
 \qquad
 \bar g:=\nabla F(x)-g^\dagger,
 \qquad
 R:=\sum_{i=1}^n
 \bigl[\nabla f_{\pi(i)}(x_{i-1})-\nabla f_{\pi(i)}(x)\bigr].
\label{eq:prox-R-definition}
\end{equation}
Then \(x_n=x-h\nabla F(x)-\eta R\).
Comparing the proximal output with the fixed point \eqref{eq:composite-fixed-point} gives
\begin{equation}
 \norm{x^+-x^\dagger}
 \le\norm{e-h\bar g}+\eta\norm R.
\label{eq:prox-confinement-comparison}
\end{equation}
By Lemma \ref{lem:pairwise-interpolation}, the first term is at most \((1-\mu h/2)\norm e\), while the inner bounds give \(\eta\norm R\le2h^2L^2D_{\mathrm c}\).
Since \(\norm e\le D_{\mathrm c}\) and \(2h^2L^2\le\mu h/2\) under \eqref{eq:prox-stepsize}, the right-hand side of \eqref{eq:prox-confinement-comparison} is at most \(D_{\mathrm c}\).
Induction over epochs completes the proof.
\end{proof}

\begin{proof}[\textbf{{Proof of Lemma \ref{lem:prox-one-epoch}}}]
The vector \(R\) is the sum of the differences \(\nabla f_{\pi(i)}(x_{i-1})-\nabla f_{\pi(i)}(x)\).
All estimates for this vector are established before the proximal map is applied, so they remain valid in the composite analysis.
For completeness, conditional on the epoch start, write
\begin{equation}
 \E R=b+r_{\mathrm{lab}}+r_{\mathrm{curv}}^{(\nu)}
       +\E R^{\mathrm{eval}},
\label{eq:prox-mean-decomposition}
\end{equation}
where
\begin{align}
 b&:=\sum_{m=0}^{n-1}
 \bigl[\nabla F(x-\eta mg)-\nabla F(x)\bigr],
\label{eq:prox-b-def}\\
 \|b\|
 &\le\eta L\frac{n(n-1)}2\|g\|,
 &
 \|r_{\mathrm{lab}}\|
 &\le n\eta L\sqrt{2V},
\label{eq:prox-mean-basic-bounds}\\
 \|r_{\mathrm{curv}}^{(\nu)}\|
 &\le c_\nu\rho_\nu\eta^{1+\nu}
       n^{(3+\nu)/2}V^{(1+\nu)/2},
\label{eq:prox-curvature-bound}\\
 \E\|R^{\mathrm{eval}}\|^2
 &\le\frac12L^4\eta^4
       \bigl(n^6\|g\|^2+n^5V\bigr),
 &
 \E\|R\|^2
 &\le L^2\eta^2
       \bigl(n^4\|g\|^2+n^3V\bigr).
\label{eq:prox-path-second-bounds}
\end{align}
Here \(r_{\mathrm{lab}}\) is the error caused by dependence of \(\pi(i)\) on the preceding labels, \(r_{\mathrm{curv}}^{(\nu)}\) is the H\"older Taylor remainder after the centered linear term cancels, and \(R^{\mathrm{eval}}\) is the sum of the differences between evaluating the current component gradient at \(x_{i-1}\) and at \(\widetilde x_{i-1}\).
These are exactly the bounds in Lemma \ref{lem:coupling}, \ref{lem:holder-curvature}, \ref{lem:B} and \ref{lem:path-energy}.
Non-expansiveness and \eqref{eq:composite-fixed-point} yield \(\norm{x^+-x^\dagger}^2\le\norm{e-h\bar g-\eta R}^2\).
After taking conditional expectation, use Lemma \ref{lem:pairwise-interpolation} and \(2h\eta\ip{\bar g}{\E R} \le h^2\norm{\bar g}^2+\eta^2\E\norm R^2\) to obtain
\begin{align}
 \E\norm{x^+-x^\dagger}^2
 \le{}&(1-\mu h)\norm e^2
 -h\left(\frac1L-2h\right)\norm{\bar g}^2
 +2\eta^2\E\norm R^2
 -2\eta\ip e{\E R}.
\label{eq:prox-master}
\end{align}

Insert \eqref{eq:prox-mean-decomposition} and use \(g=\bar g+g^\dagger\).
Cauchy--Schwarz, Jensen's inequality, and Young's inequality give the following four bounds:
\begin{align}
2\eta\abs{\ip eb}
&\le h^2L^2\norm e^2+\frac{\mu h}{32}\norm e^2
     +8\frac{L^2}{\mu}h^3\beta_\star^2,
\label{eq:prox-b-cross}\\
2\eta\abs{\ip e{r_{\mathrm{lab}}}}
&\le\frac{\mu h}{16}\norm e^2
     +32\frac{L^2V}{\mu}n\eta^3,
\label{eq:prox-label-cross}\\
2\eta\abs{\ip e{r_{\mathrm{curv}}^{(\nu)}}}
&\le\frac{\mu h}{16}\norm e^2
     +\frac{16c_\nu^2\rho_\nu^2}{\mu}
       n^{2+\nu}\eta^{3+2\nu}V^{1+\nu},
\label{eq:prox-curvature-cross}\\
2\eta\abs{\ip e{\E R^{\mathrm{eval}}}}
&\le\frac{\mu h}{16}\norm e^2
 +\frac{16L^4h^5}{\mu}\bigl(\norm{\bar g}^2+\beta_\star^2\bigr)
 +\frac{8L^4}{\mu}n^4\eta^5V.
\label{eq:prox-path-cross}
\end{align}
For example, \eqref{eq:prox-curvature-cross} follows by applying \(2ab\le\lambda a^2+b^2/\lambda\) to \eqref{eq:prox-curvature-bound} with \(\lambda=\mu h/16\); the other three lines are identical applications to the corresponding bounds in \eqref{eq:prox-mean-basic-bounds}--\eqref{eq:prox-path-second-bounds}.
The full second moment also satisfies
\begin{equation}
2\eta^2\E\norm R^2
\le4L^2h^4\bigl(\norm{\bar g}^2+\beta_\star^2\bigr)
   +2L^2n^3\eta^4V.
\label{eq:prox-second-moment-in-recursion}
\end{equation}

It remains only to check absorption.
Put \(s=hL\) and \(r=\mu/L\).
Since \(s\le r/64\), the coefficient of \(\norm{\bar g}^2\), after division by \(h/L\), is bounded by
\[
 -(1-2s)+4s^3+\frac{16s^4}{r}<0.
\]
The positive distance corrections are at most \(15\mu h\norm e^2/64\), so they leave the contraction factor no larger than \(1-\mu h/2\).
The remaining variance contribution from \(R^{\mathrm{eval}}\) is at most \(L^2n^3\eta^4V/8\), and the terms containing \(\beta_\star^2\) are at most \(9L^2h^3\beta_\star^2/\mu\).
Substituting these estimates into \eqref{eq:prox-master} proves \eqref{eq:prox-one-epoch}.
\end{proof}

\begin{proof}[\textbf{{Proof of \cref{thm:prox-global}}}]
By Lemma \ref{lem:prox-confinement}, every epoch starts in the radius- \(D_{\mathrm c}\) ball and its centered variance is at most \(G_{\mathrm c}^2\).
Apply Lemma \ref{lem:prox-one-epoch} conditionally on the epoch history and use the tower property.
The resulting affine recursion has contraction \(1-\mu h/2\); summing it geometrically with multiplier at most \(2/(\mu h)\) gives \eqref{eq:prox-global}.
\end{proof}

\begin{proof}[\textbf{{Proof of Corollary \ref{cor:prox-rate}}}]
The epoch condition implies \eqref{eq:prox-stepsize}.
Substitute \eqref{eq:prox-horizon-eta} into \eqref{eq:prox-global}.
For \(\nu\ge1/2\), the H\"older term is dominated by \(n^2T^{-3}\) by the same calculation as in Corollary \ref{cor:holder-sc-rate}.
\end{proof}

\begin{proof}[\textbf{{Proof of Corollary \ref{cor:prox-objective-certificate}}}]
For fixed \(y\), the point \(\widehat y=\operatorname{prox}_{\psi/L}(y-L^{-1}\nabla F(y))\) minimizes
\[
 Q_y(z):=F(y)+\ip{\nabla F(y)}{z-y}
          +\frac L2\norm{z-y}^2+\psi(z).
\]
Smoothness gives \(\mathcal P(\widehat y)\le Q_y(\widehat y)\), model optimality gives \(Q_y(\widehat y)\le Q_y(x^\dagger)\), and convexity of \(F\) bounds the linearization at \(x^\dagger\) by \(F(x^\dagger)\).
This proves \eqref{eq:prox-objective-certificate}.
\end{proof}

\subsection{Proximal-splitting lower bounds}

\begin{proof}[\textbf{{Proof of Lemma \ref{lem:prox-ratio}}}]
Let \(S=\sum_{j=0}^{n-1}(1-a)^j=(1-q)/a\).
Since \(S\le n\), the left-hand side is at least \((n-S)/n\).
If \(na\le1\), the binomial bound \((1-a)^j\le1-ja+\binom j2a^2\) implies \(1-(1-a)^j\ge ja/2\), and summing gives \((n-S)/n\ge na/8\).
If \(na>1\), then for every \(j\ge\lceil n/2\rceil\), \(1-(1-a)^j\ge1-e^{-1/2}\); summing over these indices gives \((n-S)/n>1/8\).
\end{proof}

\begin{proof}[\textbf{{Proof of \cref{thm:prox-lower}}}]
Choose \(c>n\lambda/\mu\) and let every component be
\begin{equation}
 f_i(u,v)=\frac\mu2u^2+\frac\mu2(v-c)^2,
 \qquad
 \psi(u,v)=\lambda\abs v.
\label{eq:prox-lower-instance}
\end{equation}
Then \(x^\dagger=(0,c-\lambda/\mu)\), \(\nabla F(x^\dagger)=(0,-\lambda)\), and all component gradients coincide.
Initialize \(y_0=(1,c-\lambda/\mu)\), and set \(a=\mu\eta\), \(q=(1-a)^n\).
The positive branch of soft thresholding is invariant.
Indeed, if \(v_k\ge c-n\lambda/\mu\), then the next post-proximal coordinate is at least \(c-(n\lambda/\mu)(q+a)\ge c-n\lambda/\mu>0\), because \(q\le1-a\).
The claim holds at initialization and hence by induction.
If \(e_k=v_k-(c-\lambda/\mu)\), the two coordinates therefore satisfy the exact recursions
\begin{equation}
 u_{k+1}=qu_k,
 \qquad
 e_{k+1}=qe_k+\frac\lambda\mu(1-q-na),
 \qquad
 u_0=1,\quad e_0=0.
\label{eq:prox-lower-recursion}
\end{equation}
Hence
\begin{equation}
 u_K=q^K,
 \qquad
 e_K=-\frac\lambda\mu(1-q^K)
       \left(\frac{na}{1-q}-1\right).
\label{eq:prox-lower-solution}
\end{equation}
If \(naK\le1/4\), Bernoulli's inequality gives \(u_K\ge3/4\); this is stronger than \eqref{eq:prox-lower-distance} because \(\lambda\le\mu\) and \(K\ge1\).
If \(naK>1/4\), then \(1-q^K\ge1-e^{-1/4}\).
Moreover, \(\min\{na,1\}\ge1/(4K)\), so Lemma \ref{lem:prox-ratio} gives \(na/(1-q)-1\ge1/(32K)\).
The second coordinate proves \eqref{eq:prox-lower-distance}.
Finally, on the positive branch the linear increment of \(\lambda v\) cancels the linear increment of the shifted quadratic, leaving exactly the quadratic gap in \eqref{eq:prox-lower-objective-equality}.
\end{proof}

\begin{proof}[\textbf{Proof of Corollary \ref{cor:prox-combined-lower}}]
Let \(\{f_i^{\rm rr}\}_{i=1}^n\) be the quadratic finite-sum hard instance from Safran and Shamir~\cite{safran2020}, including the direct-sum coordinates used there to cover the relevant constant-stepsize regimes.
Rescale this block by fixed numerical factors, if necessary, so that its smoothness and strong convexity parameters are compatible with those of the two-dimensional splitting instance in Theorem~\ref{thm:prox-lower}; this changes only the fixed condition and heterogeneity constants suppressed in the statement.
On the product space, define
\[
 \widetilde f_i(z,u,v):=f_i^{\rm rr}(z)
   +\frac{\mu}{2}u^2+\frac{\mu}{2}(v-c)^2,
 \qquad
 \widetilde\psi(z,u,v):=\lambda|v|.
\]
The proximal map acts only on the \(v\)-coordinate.
Moreover, the splitting component is identical for every \(i\), so sharing one permutation between the two blocks leaves its recursion unchanged.
Consequently, for every constant stepsize in the stated range, the \(z\)-coordinates follow exactly the Safran--Shamir random-reshuffling instance, while the \((u,v)\)-coordinates follow exactly the deterministic recursion in \eqref{eq:prox-lower-recursion}.
The product minimizer is the Cartesian product of the two block minimizers.
The gradient of the reshuffling block vanishes at its minimizer, so the product instance still satisfies \(\beta_\star=\lambda\).
Moreover,
\[
 \|y_K-x^\dagger\|^2
 =\|z_K-z_\star\|^2
  +\|(u_K,v_K)-(u_\star,v_\star)\|^2.
\]
The Safran--Shamir lower bound is stated in objective value and has the form
\[
 c_{\rm rr}\min\left\{
 \lambda_{\rm rr},
 \frac{G_{\rm rr}^2}{\lambda_{\rm rr}}
 \left(
 \frac1{(nK)^2}+\frac1{nK^3}
 \right)
 \right\}
\]
for a universal constant \(c_{\rm rr}>0\); see Safran and Shamir~\cite[Theorem~5]{safran2020}.
Since the hard block is a smooth quadratic, smoothness gives
\[
 F_{\rm rr}(z_K)-F_{\rm rr}(z_\star)
 \le \frac{L_{\rm rr}}2\|z_K-z_\star\|^2,
\]
so the same construction also gives the required squared-distance lower bound up to a fixed factor.
The product objective gap decomposes across the two blocks because both the smooth part and the regularizer are block separable.
Under \eqref{eq:prox-large-horizon-regime}, the second argument of the minimum is active.
Adding the reshuffling and splitting lower bounds therefore proves \eqref{eq:prox-combined-lower} for both squared distance and objective gap.
\end{proof}

\bibliographystyle{plain}
\bibliography{references}

\end{document}